\documentclass[a4paper]{article}

\usepackage[utf8]{inputenc}
\usepackage[english]{babel}
\usepackage{amsmath, amsthm, amsfonts, amssymb, enumerate}
\usepackage{mathrsfs}
\usepackage[numbers]{natbib}
\usepackage{color}
\usepackage{graphicx}
\usepackage{hyperref}

\theoremstyle{plain}

\newtheorem{theorem}{Theorem}
\newtheorem{lemma}{Lemma}

\newtheorem*{lemma*}{Lemma}
\newtheorem{corollary}{Corollary}

\newtheorem{maintheorem}{Theorem}

\newtheorem*{theorem*}{Theorem}
\newtheorem{proposition}{Proposition}

\theoremstyle{definition}
\newtheorem{definition}{Definition}

\newtheorem{example}{Example}

\newtheorem{remark}{Remark}

\newtheorem*{remark*}{Remark}
\newtheorem*{notation*}{Notation}

\newcommand{\cB}{\mathscr B}
\newcommand{\cD}{\mathscr D}
\newcommand{\cE}{\mathscr E}
\newcommand{\cF}{\mathscr F}
\newcommand{\cH}{\mathscr H}

\newcommand{\cS}{\mathscr S}

\newcommand{\proj}{p}

\newcommand{\unpp}{\operatorname{unpp}}
\newcommand{\id}{\operatorname{id}}
\newcommand{\conv}{\operatorname{conv}}

\newcommand{\reach}{\operatorname{reach}}
\newcommand{\lip}{C^{1,1}}
\newcommand{\N}{{\mathbb N}}
\newcommand{\R}{{\mathbb R}}

\newcommand{\ang}{{\sphericalangle}}

\newcommand{\exclude}[1]{}
\definecolor{grey}{rgb}{0.7, 0.7, 0.7}
\definecolor{darkgreen}{rgb}{0., 0.6, 0.}

\makeatletter
\def\blfootnote{\xdef\@thefnmark{}\@footnotetext}
\makeatother

\newcommand*\samethanks[1][\value{footnote}]{\footnotemark[#1]}

\newcommand{\pdifferentiable}{%
Let 
$M$ be a $C^k$-submanifold of $\R^d$ with $k\ge 1$.
Then $p$ is $C^{k-1}$ on  $\cE(M)$ and $\delta_M$ is $C^k$ on 
$\cE(M)\setminus M$.  
}
\newcommand{\varthetacontinuous}{%
Let $M\subseteq \R^d$ be a $\lip$ submanifold.
Then 
\begin{enumerate}
\item the frontier function $\vartheta$
is lower semi-continuous;
\item if $M$ is $C^2$, then $\vartheta$
is continuous.
\end{enumerate}
}
\newcommand{\theoremDpformula}{%
Let $M$ be a $C^2$ submanifold.
For every $x\in \cE(M)\setminus M$  
the derivative of $p$ in $x$ is given by
\begin{align*}
Dp(x)
&=\Big(\id_{T_{p(x)}(M)}-\|x-p(x)\|L_{p(x),v}\Big)^{-1}P_{T_{p(x)}(M)}\,,
\end{align*}
where 
\( v=\|x-p(x)\|^{-1}(x-p(x))\) and $L_{p(x),v}$ 
is the shape operator in direction $v$ at $p(x)$.
For every $x\in M$  
the derivative of $p$ in $x$ is 
$Dp(x)
=P_{T_{p(x)}(M)}\,.$
}
\newcommand{\theoremprlip}{%
Let $M$ be an $m$-dimensional topological submanifold of $\R^d$ with
$M\subseteq \cE(M)$.
Then $M$ is $\lip$.
}
\newcommand{\theoremmedialaxis}{%
Let $M\subseteq \R^d$ be a closed subset. 
Let $\cH(M)$ be the set of all closed half-spaces containing $M$.
Then 
\[
M^c=\Bigg(\bigcup_{H\in  \cH(M)}H^c\Bigg) \cup \bigcup \big\{B_r(x)\colon x\in \cS(M^c)\text{ and } r=d(x,M)\big\}\,.
\] 
}

\newcommand{\varthetacontinuoustwo}{%
If $M$ is a $\lip$ submanifold and $\cS(M^c)$ is closed, then $\vartheta$ is continuous.
}
\newcommand{\repeatthm}[2]{
\begingroup
\def\thetheorem{\ref{#2}}
\begin{theorem}
#1
\end{theorem}
\addtocounter{theorem}{-1}
\endgroup
}

\title{Existence, Uniqueness and Regularity of the Projection onto Differentiable Manifolds}

\author{Gunther Leobacher\thanks{Both authors are supported by the Austrian
Science Fund (FWF): Project F5508-N26, which is part of the Special Research
Program `Quasi-Monte Carlo Methods: Theory and Applications'.}\; and Alexander Steinicke\samethanks[1]}

\date{February 2020}

\begin{document}

\maketitle

\begin{abstract}
We  investigate the maximal open domain $\cE(M)$ on which the orthogonal 
projection map $p$ onto a 
subset  $M\subseteq \R^d$
can be defined and study essential properties of $p$.  
We prove that if $M$ is a $C^1$ submanifold of $\R^d$ satisfying a 
Lipschitz condition on the tangent spaces, then 
$\cE(M)$ can be described by a lower semi-continuous frontier function.  We
show that this frontier function is continuous if $M$ is $C^2$ or if the
topological skeleton of $M^c$ is closed and we provide an example showing
that the frontier function need not be continuous in general.

We demonstrate that, for a $C^k$-submanifold $M$ with $k\ge 2$,
the projection map is $C^{k-1}$ on $\cE(M)$, and  we obtain  a differentiation formula for the projection map which is used to  discuss boundedness of its higher order  derivatives on tubular neighborhoods.

 A sufficient condition for the inclusion $M\subseteq\cE(M)$ is that $M$ is a $C^1$ submanifold whose tangent spaces satisfy a local Lipschitz condition. We prove in a new way that this condition is also necessary. More precisely, if $M$ is a topological submanifold with $M\subseteq\cE(M)$, then $M$ must be $C^1$ and its tangent spaces satisfy the same local Lipschitz condition.

A final section is devoted to highlighting some  relations between $\cE(M)$ and the topological skeleton of $M^c$. 
\end{abstract}



\section{Introduction}

In many problems from analysis and numerical analysis it is important to know
the regularity of the orthogonal projection $p$  onto a sufficiently regular
submanifold $M$ of $\R^d$ as well as the regularity of the respective distance function.
A classical
example is the Dirichlet problem for quasilinear partial
differential equations, where the manifold of interest is the
boundary of the underlying domain \cite{serrin,Gilbarg}. 
The regularity properties of the distance function are crucial in the Lemma of Hopf, see \citet{lions} and also \citet{hopf}.
Another more recent instance is the study
of stochastic differential equations with discontinuous drift and their
numerical solvability in \citet{sz2016b, sz2017c} or \citet{sz2019}, where the manifold of interest
is the set of discontinuities of the drift coefficient.

Early results are limited to the restriction of $p$ to a small 
neighborhood of $M$, such as the tubular neighborhood theorems in \citet{hirsch}, \citet{federer1957,Foote1984}, or \citet{krantz1981}.
The first non-local result we could find is contained in 
 the article of \citet{dudekholly}, where the regularity 
of the distance function and the non-linear orthogonal projection is
studied on the maximal open domain of definition $\cE(M)$ of the projection
function.
The manifolds considered by Dudek and Holly are required to be $C^1$
with a local Lipschitz condition on the tangent bundle, referred to as 
$\lip$
(see \autoref{def:lip}). Under these 
assumptions they showed that $M\subseteq \cE(M)$ (we restate their result in
\autoref{theorem:dudek-holly}).

The four major results of our article  contribute to an almost complete understanding 
of the conditions on the submanifold for existence and regularity of the 
projection map and the shape of its domain, and will be presented subsequently.
The first of these results concerns the shape of $\cE(M)$ for a given
submanifold, which can be described by the frontier function $\vartheta$ for $M$. The latter  measures  locally how far one can move from a point of $M$ in orthogonal direction without leaving $\cE(M)$ 
(see \autoref{def:nustar}), thus quantifying the relation  $M\subseteq \cE(M)$.
\repeatthm{\varthetacontinuous}{theorem:theta-continuous}
One important consequence of
\autoref{theorem:theta-continuous} is \autoref{lemma:fiberbundle}, which states  that $\cE(M)$ is a fiber bundle if $M$ is
$C^2$.
A counterexample, \autoref{ex:theta-discont}, shows that 
if $M$ is $\lip$ but not $C^2$, then, in general, it only admits a lower semi-continuous frontier function.

The second major result is that the Lipschitz condition 
on the tangent bundle used by \citet{dudekholly} is not only sufficient, but
even necessary. More precisely, we show in \autoref{theorem:pr-lip} that a
topological submanifold 
which satisfies $M\subseteq \cE(M)$ is necessarily 
$\lip$.
\repeatthm{\theoremprlip}{theorem:pr-lip}
\autoref{theorem:pr-lip} has already already been proven in \citet{Lytchak2005} for submanifolds of Riemannian manifolds. The proof uses methods and results 
from $\operatorname{CAT}(k)$ space theory.
Using another method of proof,  \citet{Rataj2017} provide a proof of \autoref{theorem:pr-lip} for submanifolds of the $\R^d$.
We provide a new self-contained proof which is a nice application of the 
Borsuk-Ulam theorem (entering through an elegant proof of \autoref{lemma:wurscht}) and the Brouwer fixed-point theorem (entering through \autoref{lemma:extend-segment}).
We point out in \autoref{remark:blaschke} that \autoref{theorem:pr-lip}
generalizes one 
direction of the celebrated Blaschke's Rolling Theorem.

We turn to our third major result:
\citet{dudekholly}
show that the projection map is $(k-1)$-times differentiable if the 
submanifold is of class $C^k$, $k\ge 2$, thus generalizing the local 
result by \citet{Foote1984}. Their result is restated here in
\autoref{theorem:diffproj}. We give a different proof 
and
derive a differentiation formula, which in similar form 
appeared in the literature, see, e.g., \citet[Section 3]{ambman96}.
For an implicit form see \citet[Theorem 4.23]{Rataj2019}. 
\repeatthm{\theoremDpformula}{theorem:Dp-formula}
From this particular formula for the derivative we obtain
criteria for the boundedness of higher derivatives of the projection,
see \autoref{corollary4},  which weakens the requirements on 
the hypersurfaces appearing in \cite{sz2016b, sz2017c, sz2019}.
The statement of this corollary uses the concept of a subset's reach introduced in \citet{federer1957}.  \autoref{theorem:reach} highlights
connections between the reach of $M$ and its frontier function.

Our final major result is another continuity result for the frontier function of $M$, which depends on the topological skeleton of $M^c$.
\repeatthm{\varthetacontinuoustwo}{theorem:theta-continuous2}

The paper is organized as follows:
In Section \ref{sec:domainE} we recall and prove basic
 properties of the projection $p$ and the set $\cE(M)$. Some of these results
are of independent interest as they also hold for general subsets $M\subseteq\R^d$. 
The  section also contains \autoref{theorem:theta-continuous} and its corollaries.
In Section \ref{sec:regularity} we prove regularity of $p$ and the corresponding distance function
(\autoref{theorem:diffproj}) and give a differentiation formula for $p$ in terms of the 
manifold's shape operator (\autoref{theorem:Dp-formula}). We relate the 
boundedness of the (higher) derivatives of unit normal fields of a 
$C^k$ hypersurface, $k\ge 2$, with positive reach 
to the boundedness of the higher derivatives
of $p$ (\autoref{corollary4}).
Section \ref{section:converse} is dedicated to the proof of 
\autoref{theorem:pr-lip} and
finally, Section \ref{sec:skeleton} highlights the relation between $\cE(M)$ and the medial axis/topological skeleton of $M^c$ in 
\autoref{theorem:skeleton}. It contains 
a version of the medial axis transform adapted to our setting
as well as \autoref{theorem:theta-continuous2}.

\section{Parametrization of $\cE(M)$}\label{sec:domainE}

We give some basic definitions and introduce some notation used throughout the paper.

\begin{definition}\label{definition:unpp}
Let $M\subseteq \R^d$ be a nonempty set.
\begin{enumerate}[(i)]
\item  
For every point $x\in \R^d$ denote the distance between $x$ and $M$
by $d(x,M):=\inf\{\|x-\xi\|\colon \xi\in M\}$, where $\|\cdot\|$ is the Euclidean norm on $\R^d$.

We denote 
the distance function $\delta_M:\R^d\to {[0,\infty)}$ by 
$\delta_M(x)=d(x,M)$.

\item 
For  
$\varepsilon\in (0,\infty)$,
we denote the {\em $\varepsilon$-neighborhood} of $M$ by
\[
M^\varepsilon:=\{x\in \R^d: d(x,M)<\varepsilon\}\,
.\]

\item  
We define \[
\unpp(M):=\{x\in \R^d \colon
\,\exists!\, \xi\in M : \|x-\xi\|=d(x,M)\}\ .
\]
Thus there exists $p: \unpp(M)\to M$ such that 
for all $x\in \unpp(M)$ it holds that 
$p(x)$ is the unique nearest point to $x$ on $M$.
The function $p$ is called the (orthogonal) projection onto $M$.

A set $U\subseteq \R^d$ has the {\it unique nearest point property (unpp) with
respect to $M$} iff $U\subseteq \unpp(M)$.

\item Let 
\(\cE(M):=\bigcup\{U \subseteq \R^d\colon U \text{ is open and }U\subseteq \unpp(M)\}=\unpp(M)^\circ\)
be the maximal open set on which the function $p$ is defined.
\end{enumerate}
\end{definition}

\begin{notation*}[Balls]
For $x\in \R^d$ and $r\in (0,\infty)$ denote by
\begin{align*}
B_r(x)&:=\{z\in \R^d\colon \|x-z\|<r\}
\end{align*}
the open ball with center $x$ and radius $r$ and by
\begin{align*}
\bar B_r(x)&:=\{z\in \R^d\colon \|x-z\|\le r\}=\overline{B_r(x)}
\end{align*}
the closed ball. We denote the $(d-1)$-dimensional unit sphere by $S:=\bar B_1(0)\setminus B_1(0)$.
\end{notation*}

\begin{notation*}[Line segments]
For $x_1,x_2\in \R^d$ denote
\begin{align*}
{]}x_1,x_2{[}&:=\{(1-\lambda)x_1+\lambda x_2\colon \lambda\in (0,1)\}
\end{align*}
and let ${]}x_1,x_2],[x_1,x_2{[},[x_1,x_2]$ be the corresponding sets with
$(0,1)$ replaced by $(0,1],[0,1),[0,1]$, respectively.
\end{notation*}

\begin{definition}\label{defi:submanifold}
Let $d,m,k\in \N\cup\{0\}$, $m < d$,  and let $M\subseteq \R^d$.
We say $M$ is an $m$-dimensional $C^k$ {\em submanifold} of $\R^d$ iff 
for every $\xi\in M$ there exist open sets $U,V\subseteq \R^d$ 
and a $C^k$ diffeomorphism $\Psi:V\to U$ such that $\xi\in U$ and
for all $y=(y_1,\dots,y_d)\in V$ it holds
$\Psi(y)\in M \Longleftrightarrow y_{m+1}=\dots=y_d=0$. 
In the case where $k=0$, by a $C^0$ diffeomorphism we mean a homeomorphism,
and we also call $M$  a {\em topological submanifold}. For the case $k\geq 1$, we write $T_\xi(M)$ for the tangent space on $M$ in $\xi\in M$,
\[
T_\xi(M):=\{t\in \R^d \colon \exists \gamma:(-1,1)\to M \text{ a $C^1$ map with }\gamma(0)=\xi \text{ and }\gamma'(0)=t \}\,.
\]
\end{definition}

\begin{remark}
Usually, the case $m=d$ is not excluded in the definition of
a submanifold. However, for the questions considered here this
case is not very interesting: a $d$-dimensional submanifold $M\subseteq \R^d$
is an open subset of $\R^d$. Thus, no point in $\R^d\setminus M$ 
has a nearest point on $M$ and for every $x\in M$ we have $p(x)=x$,
so $\cE(M)=M$ and $p$ is $C^\infty$ on $\cE(M)$.

This means that the case $m=d$ is not interesting for the kind of questions we pursue in this article,
and we shall always assume $m<d$.
\end{remark}

\begin{remark}\label{remark:submanifold}
\begin{enumerate}
\item \label{item:remark_submanifold1}
If $M$ is a $C^k$ submanifold with $k\ge 0$, $\xi\in M$ and 
$\Psi:V\to U$ is a diffeomorphism as in \autoref{defi:submanifold},
then the map $\psi:\{y\in \R^m\colon (y_1,\dots,y_m,0,\dots,0)\in V\}\to \R^d$,
\[
\psi(y_1,\dots,y_m):=\Psi(y_1,\dots,y_m,0,\dots,0)
\]
is a (local) parametrization of $M$ with  $\xi$ in its image. We may always assume that $0\in V$ and $\xi=\psi(0)$.
\item \label{item:remark_submanifold2}
If $M$ is $C^1$ and $\xi\in M$ then, by virtue of the implicit function theorem, 
$M$ can be locally represented as the graph of
a $C^1$ function $\Phi$ from the tangent space $T_\xi(M)$ into
the corresponding normal space $T_\xi(M)^\perp$. 

More precisely,
there exist open sets  $W\subseteq T_\xi(M)$ and $U\subseteq T_\xi(M)^\perp$ 
with $0\in W\cap U$
and a $C^1$ function $\Phi:W\to U$ 
such that $\Phi(0)=0$ and 
$M\cap (\xi+W+U)=\{\xi+t+\Phi(t)\colon t\in W\}$.

One cannot generalize the statement to topological submanifolds, even if the
tangent space is replaced by some other linear space: consider as a 
counterexample $M:=\{x^2\colon x\in [0,\frac{1}{2})\}\cup \{x^3\colon x\in [0,\frac{1}{2})\}$.

\end{enumerate}
\end{remark}

The following definition corresponds to condition (3.3) in \cite{dudekholly}.
\begin{definition}\label{def:lip} Let $m\in \N\setminus \{0\}$.
Denote by $G(m,\R^d)$ the Grassmannian of
$m$-dimensional subspaces of the $\R^d$. 
For $T_1,T_2\in G(m,\R^d)$ their {\em Hausdorff distance}  is defined as
\[
d_H(T_1,T_2):=\sup\big\{\inf\{\|t_2-t_1\|\colon t_2\in T_2\cap S\}\colon t_1\in T_1\cap S\big\}\,.
\]

We say $M$ is an $m$-dimensional $\lip$ {\em submanifold} of $\R^d$ iff 
$M$ is an $m$-dimensional  $C^1$ submanifold of $\R^d$ and the map 
$M\to G(m,\R^d)$, $\xi\mapsto T_\xi(M)$ is locally Lipschitz w.r.t.~the Hausdorff distance,
i.e., if for all $\xi\in M$ there exists an open set $V\subseteq M$ and a positive 
constant $L$ such that 
$\xi\in V$ and $\forall \eta\in V\colon d_H\big(T_\xi(M),T_\eta(M)\big)\le L \|\xi-\eta\|$.
\end{definition}

\begin{definition} Let $k\ge 1$. For a $C^k$ manifold $M$ 
let \[\nu(M):=\{(\xi,v)\in \R^d\times \R^d:\xi\in M,v\perp T_\xi(M)\}\]
be the {\em normal bundle} for $M$. Moreover, let $$\nu_1(M):=\{(\xi,v)\in \nu(M)\colon \|v\|=1\}$$
and define the {\em endpoint map} $F:\nu(M)\to \R^d$ by 
$F(\xi,v):=\xi+v$.
\end{definition}

\begin{remark}\label{remark:F-Ck-1}
$\nu(M)$ is a $C^{k-1}$ manifold and thus $F$ is a $C^{k-1}$ function.
\end{remark}

The following result is a direct consequence of \citet[Theorem 3.8]{dudekholly}.

\begin{theorem}\label{theorem:dudek-holly}
Let $M \subseteq \R^d$ be a  $\lip$ submanifold and let 
$\xi \in M$.
Then  
$\xi$ has an open neighborhood $U$ in $\R^d$ so that $U\subseteq \unpp(M)$
and 
 for all $x\in U$, $\zeta\in U\cap M$ with $x-\zeta\perp T_\zeta(M)$ 
it holds $p(x)=\zeta$, i.e., $p(F(\zeta,x-\zeta))=\zeta$.
\end{theorem}

\begin{remark}\label{remark:first-order}
Note that \autoref{theorem:dudek-holly} implies that 
$M\subseteq \cE(M)$ if $M$ is $\lip$.

Note further that if $x\in \cE(M)\setminus M$, then 
$(x-p(x))\perp T_{p(x)}(M)$, since the sphere 
$\bar B_{x-p(x)}(x)\setminus B_{x-p(x)}(x)$ has a first order contact
with $M$ in $p(x)$.

\end{remark}

The next lemma  shows that unpp cannot be a property of isolated 
points. The lemma is further strengthened  by Lemmas 
\ref{lemma:line-segment-open} and \ref{lemma:extend-segment}. 

\begin{lemma}\label{lemma:line-segment}
Let 
$M\subseteq \R^d$, $x\in \R^d\setminus M$, and assume that there exists  
$\xi\in M$ with $\|x-\xi\|=d(x,M)$.

Then the line segment ${]}x,\xi]$ has 
the unpp w.r.t.~$M$, i.e., ${]}x,\xi]\subseteq \unpp(M)$,
and for every $z\in {]}x,\xi]$ it holds that $p(z)=\xi$.
\end{lemma}

\begin{proof}

Let $z\in{]}x,\xi]$. For $z=\xi$ the assertion is obvious.
Now consider the case $z\ne \xi$, and let $\eta\in \overline{M}$.
We have 
\begin{align}
\nonumber\|x-\xi\|&=\|x-z\|+\|z-\xi\|\\
\label{eq:triangle}\|x-\eta\|&\le\|x-z\|+\|z-\eta\|\,.
\end{align}
By the continuity of the distance function we have $\|x-\xi\|\le \|x-\eta\|$,
so that $\|z-\xi\|\le\|z-\eta\|$. 

Suppose $\|z-\xi\|=\|z-\eta\|$.
Thus equality holds in 
\eqref{eq:triangle},  which implies $z\in [x,\eta]$. 
$z\ne x$ by assumption and $z\ne \eta$ because,  $\|z-\eta\|=\|z-\xi\|>0$.
Thus $z=\lambda\eta+(1-\lambda)x$ for some $\lambda\in (0,1)$. On the other
hand $ z=\mu\xi+(1-\mu)x$ for some $\mu\in (0,1)$, so that 
\[
\eta-z=\frac{1-\lambda}{\lambda}(z-x)
\text{ and }
\xi-z=\frac{1-\mu}{\mu}(z-x)\,.
\]
From $\|z-\xi\|=\|z-\eta\|$ we conclude $\lambda=\mu$ and thus $\eta=\xi$.
\end{proof}

In \cite[Theorem 4.8]{federer1957} it is shown that for every closed set $M$
the projection map 
$p$ onto $M$ is continuous on every set where it is well-defined.
The subsequent lemma is a version for which closedness
is not needed; for every subset $M\subseteq\R^d$ the projection 
$p$ is continuous on every {\em open} set on which it
is well-defined.
The following result can be found in  \cite[Theorem 1.3]{dudekholly}.

\begin{proposition}[Continuity of $p$]\label{lemma:continuous}
Let $M\subseteq \R^d$,
$U\subseteq \R^d$ be an open set with $U\subseteq \unpp(M)$, 
and let $p \colon U\to M$ denote
the corresponding projection map.
Then $p$ is continuous.
\end{proposition}

\begin{example}\label{ex:lemma_cont}
Consider the following example of a non-compact submanifold:
\begin{center}
\begin{picture}(0,0)%
\includegraphics{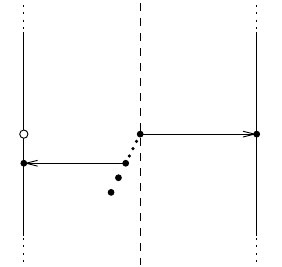}%
\end{picture}%
\setlength{\unitlength}{3947sp}%
\begingroup\makeatletter\ifx\SetFigFont\undefined%
\gdef\SetFigFont#1#2#3#4#5{%
  \reset@font\fontsize{#1}{#2pt}%
  \fontfamily{#3}\fontseries{#4}\fontshape{#5}%
  \selectfont}%
\fi\endgroup%
\begin{picture}(2243,2120)(235,-3969)
\put(2346,-2560){\makebox(0,0)[lb]{\smash{{\SetFigFont{9}{10.8}{\rmdefault}{\mddefault}{\updefault}{\color[rgb]{0,0,0}$\ell_3$}%
}}}}
\put(2346,-2938){\makebox(0,0)[lb]{\smash{{\SetFigFont{9}{10.8}{\rmdefault}{\mddefault}{\updefault}{\color[rgb]{0,0,0}$\proj(x)$}%
}}}}
\put(366,-3200){\makebox(0,0)[rb]{\smash{{\SetFigFont{9}{10.8}{\rmdefault}{\mddefault}{\updefault}{\color[rgb]{0,0,0}$\proj(x_n)$}%
}}}}
\put(1211,-3084){\makebox(0,0)[rb]{\smash{{\SetFigFont{9}{10.8}{\rmdefault}{\mddefault}{\updefault}{\color[rgb]{0,0,0}$x_n$}%
}}}}
\put(1415,-2851){\makebox(0,0)[lb]{\smash{{\SetFigFont{9}{10.8}{\rmdefault}{\mddefault}{\updefault}{\color[rgb]{0,0,0}$x$}%
}}}}
\put(250,-3666){\makebox(0,0)[lb]{\smash{{\SetFigFont{9}{10.8}{\rmdefault}{\mddefault}{\updefault}{\color[rgb]{0,0,0}$\ell_2$}%
}}}}
\put(250,-2560){\makebox(0,0)[lb]{\smash{{\SetFigFont{9}{10.8}{\rmdefault}{\mddefault}{\updefault}{\color[rgb]{0,0,0}$\ell_1$}%
}}}}
\put(2463,-3375){\makebox(0,0)[lb]{\smash{{\SetFigFont{9}{10.8}{\rmdefault}{\mddefault}{\updefault}{\color[rgb]{0,0,0}$M=\ell_1\cup \ell_2\cup \ell_3$}%
}}}}
\end{picture}%

\end{center}
Here the projection is not continuous in $x$. On the other hand, there 
is no open set $U$ containing $x$ and having the unpp.
\end{example}

\medskip 
Also the next result can be found in \cite[Theorem 1.5.(ii)]{dudekholly}.

\begin{lemma} \label{lemma:line-segment-open}
Let $M\subseteq \R^d$ and $x\in \R^d\setminus M$ such that there exists an open set $U\subseteq \unpp(M)$ containing $x$. 
Then there exists an open set $\hat U\subseteq \unpp(M)$ that contains ${]}x,p(x){[}$.
\end{lemma}

\begin{lemma} \label{lemma:extend-segment}
Let $M\subseteq \R^d$  and  
$x\in \R^d\setminus M$ such that there exists an open set 
$U\subseteq \unpp(M)$ containing $x$. 
Then there exists $a\in(1,\infty)$ such that 
\[p(x)+a(x-p(x))\in U \text{ and }p(p(x)+a(x-p(x)))=p(x)\,.\]
\end{lemma}

\begin{proof}
Consider a closed ball $B:=\bar B_\varepsilon(p(x))$ with $\varepsilon\le\frac{\sqrt{3}}{2}\|x-p(x)\|$. 
Then for every $y\in B$ it holds 
that the (unsigned) angle
$\ang(x-y,x-p(x))$ lies in the interval $[0,\frac{\pi}{3}]$. 
By the continuity of $p$ (\autoref{lemma:continuous}) there exists $\delta\in(0,\|x-p(x)\|-\varepsilon)$ such that
$\forall z\in\R^d\colon \|z-x\|\le\delta 
\Rightarrow \big(z\in U \text{ and } p(z)\in B\big)$.
Let 
\begin{align*}
\cD:=\big\{z\in \R^d\colon \exists v \in \R^d\colon z=x+\tfrac{1}{2}\delta&\|x-p(x)\|^{-1}(x-p(x))+v,\\
&\langle x-p(x), v\rangle=0  \text{ and } \|x-z\|\le \delta\big\}\,,
\end{align*}
i.e., $\cD$ is the $(d-1)$-dimensional closed ball which is orthogonal to 
$x-p(x)$, lies on the side of $x$ opposing $p(x)$, and 
has the property that, for all $z\in \cD$, the angle $\ang(z-x,x-p(x))$  lies in the interval $[0,\frac{\pi}{3}]$.

\begin{center}
\begin{picture}(0,0)%
\includegraphics{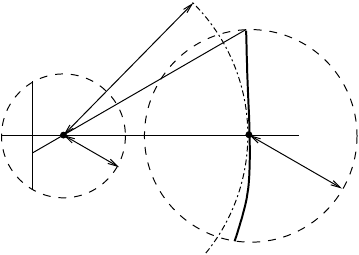}%
\end{picture}%
\setlength{\unitlength}{4144sp}%
\begingroup\makeatletter\ifx\SetFigFont\undefined%
\gdef\SetFigFont#1#2#3#4#5{%
  \reset@font\fontsize{#1}{#2pt}%
  \fontfamily{#3}\fontseries{#4}\fontshape{#5}%
  \selectfont}%
\fi\endgroup%
\begin{picture}(2729,1932)(4649,-4344)
\put(6690,-3355){\makebox(0,0)[b]{\smash{{\SetFigFont{7}{8.4}{\rmdefault}{\mddefault}{\updefault}{\color[rgb]{0,0,0}$\proj(x)$}%
}}}}
\put(5110,-3368){\makebox(0,0)[b]{\smash{{\SetFigFont{7}{8.4}{\rmdefault}{\mddefault}{\updefault}{\color[rgb]{0,0,0}$x$}%
}}}}
\put(6596,-2892){\makebox(0,0)[b]{\smash{{\SetFigFont{7}{8.4}{\rmdefault}{\mddefault}{\updefault}{\color[rgb]{0,0,0}$M$}%
}}}}
\put(4826,-3233){\makebox(0,0)[b]{\smash{{\SetFigFont{7}{8.4}{\rmdefault}{\mddefault}{\updefault}{\color[rgb]{0,0,0}$\cD$}%
}}}}
\put(6926,-3764){\makebox(0,0)[b]{\smash{{\SetFigFont{7}{8.4}{\rmdefault}{\mddefault}{\updefault}{\color[rgb]{0,0,0}$\varepsilon$}%
}}}}
\put(5274,-3627){\makebox(0,0)[b]{\smash{{\SetFigFont{7}{8.4}{\rmdefault}{\mddefault}{\updefault}{\color[rgb]{0,0,0}$\delta$}%
}}}}
\put(5558,-2701){\makebox(0,0)[b]{\smash{{\SetFigFont{7}{8.4}{\rmdefault}{\mddefault}{\updefault}{\color[rgb]{0,0,0}$\|x-\proj(x)\|$}%
}}}}
\end{picture}%

\end{center}

Therefore, every line spanned by
$y\in B$ and $x$ has precisely one intersection with $\cD$.
Define this as $f(y)$ and consider the mapping $g:B\to B$ defined
by $g(y):=p(f(y))$. Note that $f$ is continuous, so $g$ is
a continuous mapping from $B\to B$. Since $B$ is homeomorphic
to the unit ball in $\R^d$, there exists a fixed point $y_0$ of $g$
by Brouwer's fixed
point theorem, 
i.e., there exists $y_0\in B$ such that $g(y_0)=y_0$. 
Note that $y_0\in M$ by the definition of $g$.

Now $p(f(y_0))=y_0$ and $x\in {]}f(y_0),y_0{[}$. By 
\autoref{lemma:line-segment} we therefore have $p(x)=y_0$
and thus $p(f(y_0))=p(x)$. 

We conclude the proof by noting that $f(y_0)$ is of 
the desired form $f(y_0)=p(x)+a(x-p(x))$ with $a\in (1,\infty)$.
\end{proof}

\begin{remark}
Let us revisit Example 1. The point $x$ lies in $\unpp(M)$
and $]x,p(x)[\subseteq \cE(M)$. However, $x\notin \cE(M)$, and also
there does not exist
$a\in(1,\infty)$ such that 
$p(p(x)+a(x-p(x)))=p(x)$.
Therefore, the assumption made in \autoref{lemma:extend-segment},
that $x$ be contained in some open set in $\unpp(M)$, is necessary.
\end{remark}

\begin{example}
Consider the set $M:=\{(x,x^2)\in \R^2\colon x\ge 0\}$.
Routine calculations yield that 
\begin{align*}
\cE(M)&=\R^2\setminus\Big\{(x,y)\colon x\le 0,\ 
y=\tfrac{1+3(x^2)^\frac{1}{3}}{2}\Big\}\,,\\
\unpp(M)&=\cE(M)\cup\{(0,\tfrac{1}{2})\}\,.
\end{align*}
The point $(0,\frac{1}{2})$ has a unique closest point on 
$M$, namely $(0,0)$, but does not lie in $\cE(M)$.
Of course, for all $y\in [0,\frac{1}{2}]$ we have $p\big((0,y)\big)=(0,0)=p\big((0,\frac{1}{2})\big)$,
but for $y\in \big[\frac{1}{2},\infty)$ it holds
$p\big((0,y)\big)=\big(\sqrt{\frac{2y-1}{4}},\frac{2y-1}{4}\big)$.

In this example $p$ is continuous on 
$\unpp(M)\ne \cE(M)$, but still
the conclusion of \autoref{lemma:extend-segment} does not hold
since there is no open set $U\subseteq \unpp(M)$ with $x\in U$.
\end{example}

\begin{definition}\label{def:nustar}
Let $M$ be a $\lip$ submanifold of $\R^d$. 
\begin{enumerate}
\item  We define 
\begin{align*}
\nu^*(M)&:=\big\{(\xi,v)\in \nu(M)\colon \big{[}\xi,\xi+v\big{]}\subseteq \cE(M)\big\}\,,
\end{align*}
and $F^*:\nu^*(M)\to \cE(M)$  by $F^*(\xi,v):=\xi+v$,
the restriction of the endpoint map to $\nu^*(M)$.
\item 
We define $\vartheta\colon\nu_1(M)\to (0,\infty],$ $\vartheta(\xi,v):=\sup\{r> 0\colon {]}\xi,\xi+rv{[}\subseteq\cE(M)\}$ and call $\vartheta$ the {\em frontier function} for $M$.
\end{enumerate}
\end{definition}

\begin{proposition}\label{proposition:thetanu}
Let $M$ be a $\lip$ submanifold of $\R^d$. Then
\begin{enumerate}
\item\label{item:thetanu1} the frontier function $\vartheta$ is well-defined,

\item\label{item:thetanu2} for all $(\xi,v)\in \nu_1(M)$ and all $r\in \big[0,\vartheta(\xi,v)\big)$ 
it holds $p(\xi+rv)=\xi$,
\item\label{item:thetanu4} \(
\nu^*(M)=\big\{(\xi,rv)\colon(\xi,v)\in\nu_1(M), r\in\big[0,\vartheta(\xi,v)\big)\big\}\,,
\) 
\item\label{item:thetanu5}  for every $x\in \cE(M)$ we have $[x,p(x)]\subseteq\cE(M)$.
\end{enumerate}
\end{proposition}
\begin{proof}
\noindent{\em\ref{item:thetanu1}.} By \autoref{theorem:dudek-holly}, for all $\xi \in M$ there exists
an open set $U\subseteq \unpp(M)$ containing $\xi$.
Since $U$ is open and $U\subseteq \cE(M)$, 
the set $\big\{r\in(0,\infty)\colon {]}\xi,\xi+rv{[}\subseteq\cE(M)\big\}$ is non-empty and
therefore $\vartheta(\xi,v)>0$ for all $(\xi,v)\in \nu_1(M)$.

\noindent{\em\ref{item:thetanu2}.}  
Clearly  $p(\xi)=\xi$.  
Now let $H:=\big\{r\in(0,\infty)\colon p(\xi+r v)=\xi\big\}$.
The set $H$ is non-empty by \autoref{theorem:dudek-holly}, 
so  $s:=\sup H$ is a positive real number or equal to infinity. Assume $s<\vartheta(\xi,v)$. Then 
$\xi+s v\in {\big]\xi,\xi+\frac{1}{2}(s+\vartheta(\xi,v))v\big[}\subseteq \cE(M)$,  in particular $\xi+s v\in \cE(M)$. Since $p(\xi+rv)=\xi$ for all 
$r\in (0,s)$ and  $p$ is 
continuous on $\cE(M)$, it holds $p(\xi+s v)=\xi$.
By \autoref{lemma:extend-segment} there exists $a\in (1,\infty)$ such 
that 
\[\xi + asv\in \cE(M) \text{ and }p\big(\xi+ asv\big)=\xi\,.\]
Hence, by \autoref{lemma:line-segment-open} also
$]\xi, \xi+ asv[\subseteq\cE(M)$, contradicting $s:=\sup H$.
Therefore, $s\ge \vartheta(\xi,v)$ and $p(\xi+r v)=\xi$ for all $r\in (0,\vartheta(\xi,v))$.


\noindent{\em\ref{item:thetanu4}.} This is obvious.

\noindent{\em\ref{item:thetanu5}.} 
By \autoref{lemma:line-segment-open} we have that for every $x\in \cE(M)$ 
also the line segment $[x,p(x){[}$ is contained in $\cE(M)$. By
\autoref{theorem:dudek-holly} we have $M\subseteq \cE(M)$ so that indeed
$[x,p(x)]\subseteq\cE(M)$.
\end{proof}

\begin{lemma}\label{right_projection_lemma}
Let $M$ be a $\lip$ submanifold.
Then $F^*$ is a homeomorphism and $p(F^*(\xi,v))=\xi$
for all $(\xi,v)\in \nu^*(M)$. 
\end{lemma}

\begin{proof}
 1. $F^*$ is injective: 
Let $(\xi,v),(\zeta,w)\in \nu^*(M)$ with $\xi+v=F^*(\xi,v)=F^*(\zeta,w)=\zeta+w$. 
From item \eqref{item:thetanu2} of \autoref{proposition:thetanu} it follows $p(\xi+v)=\xi$ and $p(\zeta+w)=\zeta$.
By the unpp of $\cE(M)$, it holds $\xi=p(\xi+v)=p(\zeta+w)=\zeta$.
Together with  $\xi+v=\zeta+w$ we also get $v=w$.

 2. $F^*$ is surjective:
since every $x\in \cE(M)$ can by written as $x=p(x)+(x-p(x))$
and $(x-p(x))\perp T_{p(x)}(M)$ (see \autoref{remark:first-order}), 
we have $x=F^*(p(x),x-p(x))$.
By \autoref{proposition:thetanu}.(\ref{item:thetanu5}), 
${[p(x),x]}\subseteq \cE(M)$. 

 3. The function $F^*$ is clearly continuous. Its inverse 
satisfies $(F^*)^{-1}(x)=(p(x),x-p(x))$, and it 
is continuous since $p$ is continuous by \autoref{lemma:continuous}. 
\end{proof}

We recall some well-known concepts.

\begin{definition}\label{def:unit-normal-field}
Let  $M$ be an $m$-dimensional $C^1$ submanifold of
the $\R^d$.
\begin{enumerate}
\item Let $V\subseteq M$ be an open set and let 
$n\colon V\to \R^d$ be a continuous function
such that $\|n(\eta)\|=1$ and 
$n(\eta)\in(T_\eta(M))^\perp$ for every $\eta\in V$. 
Then we call $\big(V,n\big)$  a {\em unit normal field}.
\item Let $V\subseteq M$ be an open set and let 
$n_{m+1},\dots,n_d\colon V\to \R^d$ be continuous functions
such that 
\[\langle n_j(\eta), n_\ell(\eta)\rangle=
\begin{cases}
1& j=\ell\\
0& j\ne \ell
\end{cases}
\] and 
$n(\eta)\in(T_\eta(M))^\perp$ for every $\eta\in V$. 
Then we call $\big(V,n_{m+1},\dots,n_d\big)$  an {\em orthonormal moving 
frame of $\nu(V)$}.
\end{enumerate}
\end{definition}

It is not hard to show -- using the subsequent lemma and induction -- 
that if $M$ is a $C^k$ submanifold with $k\ge 1$,
then for every $\xi\in M$ there exists a $C^{k-1}$ orthonormal moving 
frame $\big(V,n_{m+1},\dots,n_d\big)$ of $\nu(V)$ with $\xi\in V$.

\begin{proposition}\label{lemma:normal-vectors}
Let 
$M$ be an $m$-dimensional $C^k$ submanifold of the $\R^d$, with $k\ge 1$. 
Then for every $(\xi,v)\in \nu_1(M)$ there exists 
a $C^{k-1}$ unit normal field $\big(V,n\big)$ 
with $\xi\in V$ and $n(\xi)=v$.

For $k\ge 2$ it holds that, if $(V_1,n_1)$ is another unit normal field
with $\xi\in V_1$ and $n_1(\xi)=v$,
then\[
P_{T_\xi(M)} Dn(\xi)=P_{T_\xi(M)} Dn_1(\xi)\,,
\]
where $P_{T_\xi(M)}$ is the projection onto the tangent space $T_\xi(M)$.
\end{proposition}

\begin{proof}
Let $(\xi,v)\in \nu_1(M)$.
Choose a local parametrization $\psi:W\to \R^d$  
of $M$  with $0\in W$ and 
$\psi(0)=\xi$ (see item \ref{item:remark_submanifold1} of \autoref{remark:submanifold}). For every $y\in W$ and every $j\in\{1,\dots,m\}$ define 
$t_j(y):=\frac{\partial}{\partial y_j}\psi(y)$.
Note that for every $y\in W$ the set $\{t_1(y),\dots,t_d(y)\}$ forms a 
basis of the tangent space $T_{\psi(y)}(M)$. 

If $m=d-1$, then the cross product $w:=t_1\times\dots\times t_d$
is normal to $M$ and $w(\xi)=\lambda v$ for some $\lambda\in \R\setminus \{0\}$.
W.l.o.g., $\lambda>0$. Now the vector field $n=\|w\|^{-1}w$ is a
unit normal field on $V=\psi(W)$ with $\xi\in V$ and $n(\xi)=v$.

Now consider the case $m<d-1$. Write $v_{m+1}:=v$ and
extend $v_{m+1}$ to a basis  $v_{m+1},v_{m+2},\dots, v_d$ of 
$\big(T_{\psi(0)}(M)\big)^\perp$. Then 
$$\det\big( t_1(0),\dots,t_m(0),v_{m+1},\dots, v_d\big)\ne 0,$$ and by the 
continuity of the determinant
and the functions $t_1,\dots, t_m$  there exists $c\in (0,\infty)$
and an open set $W_1\subseteq W$ 
containing $0$  such that for all $y\in W_1$ we have
$\big|\det\big(t_1(y),\dots,t_d(y),v_{m+1},\dots, v_d\big)\big|\ge c.$

Denote by $P(y)$ the orthogonal projection from $\R^d$ 
onto the space spanned by $\{t_1(y),\dots,t_m(y)\}$, i.e.,  on 
$T_{\psi(y)}(M)$, and
define $n(\psi(y))$ by
\begin{align*}
n(\psi(y))&:=\|v_{m+1}-P(y)v_{m+1}\|^{-1} (v_{m+1}-P(y)v_{m+1})\,.
\end{align*}

Finally, in both cases, $V:=\psi(W_1)$ is an open subset of $M$ by the invariance of domain theorem and we have $\xi=\psi(0)\in V$. So $(V,n)$ is a
unit normal field with $\xi\in V$ and $n(\xi)=v$.

Let 
$(V_1,n_1)$ be another unit normal field with $\xi\in V_1,n_1(\xi)=v$.
\begin{align*}
P_{T_\xi(M)}D n(\xi)-P_{T_\xi(M)}D n_1(\xi)
&=P_{T_\xi(M)}D (n- n_1)(\xi)\,.
\end{align*}
For every $C^1$ vector field $(V_2,t)$, with $t\colon V_2\subset V\cap V_1
\to \R^d$ 
satisfying $t(\zeta)\in T_\zeta(M)$ for all $\zeta$,  we have 
\begin{align*}
\langle n-n_1,t\rangle&=0\\
t^\top D (n-n_1)+ (n-n_1)^\top D t&=0\,.
\end{align*}
Since $(n-n_1)(\xi)=0$, we have $t^\top D (n-n_1)(\xi)=0$. Because $t$ was arbitrary the result follows.
\end{proof}

\begin{definition}\label{definition:shape}
Let $M$ be a $C^1$ submanifold of $\R^d$ and $(\xi,v)\in \nu_1(M)$.
\begin{enumerate}
\item For an arbitrary $C^1$ unit normal field $(V,n)$ with $\xi\in V$,
$n(\xi)=v$ we define the 
{\em shape operator}
$L_{\xi,v}:T_\xi(V)\to T_\xi(V)$ by
\[
L_{\xi,v}:= -P_{T_\xi(M)}  Dn(\xi)\,,
\]
where $P_{T_\xi(M)}$ denotes the orthogonal projection onto the tangent space.
Note that, $L_{\xi,v}$ is well-defined by \autoref{lemma:normal-vectors}.
\item Denote by $\lambda_1,\dots,\lambda_\ell$
the (not necessarily distinct) positive eigenvalues of  $L_{\xi,v}$.
Then the points $\xi+\lambda_1^{-1}v,\dots,\xi+\lambda_\ell^{-1} v$ are called 
{\em centers of curvature} of $M$ in $\xi$ in direction of $v$.

\item For every $(\xi,v)\in \nu_1(M)$ 
denote by $\varrho(\xi,v)$
the {\em radius of
curvature} of $M$ at $\xi$ in direction of the unit normal $v$,
\[
\varrho(\xi,v):=\inf\{r\in (0,\infty)\colon \forall \tau\in (0,\infty)\colon
B_\tau(\xi)\cap M\cap B_r(\xi+rv)\ne\emptyset\}\,,
\]
with the convention that $\inf \emptyset=\infty$.
\end{enumerate}
\end{definition}

The following fact is most likely folklore, yet it is not easy 
to find a citation for item \eqref{item:radiusiseigen}. Therefore, a proof is provided in the appendix.

\begin{proposition}\label{theorem:curvature}
Let $M\subseteq \R^d$ be a $C^2$ submanifold 
and let $(\xi,v)\in \nu_1(M)$. Then
\begin{enumerate}
\item The shape operator $L_{\xi,v}$ is  self-adjoint.
\item\label{item:radiusiseigen} 
Denote by $\lambda_1,\dots,\lambda_m$
the (not necessarily distinct) eigenvalues of  $L_{\xi,v}$.

Then  
\[
\varrho(\xi,v)=\big(\max(0,\max(\lambda_1,\dots,\lambda_m))\big)^{-1}\,,
\]
with the convention that $0^{-1}=\infty$.
\end{enumerate}
\end{proposition}

The next proposition 
characterizes the critical values of the endpoint map. It
is the second example in Section 1.3 in 
\cite{arnoldsymplectic} and also follows from \cite[4.1.9 Corollary]{PalaisTang}.

\begin{proposition}\label{lemma:arnold}
Let $M\subseteq \R^d$ be a $C^2$ submanifold and let $(\xi,v)\in \nu(M)$.
Then $\big(\det F'\big)(\xi,v)=0$ 
iff $\xi+v$ is a center of curvature in $\xi$ in direction of $\|v\|^{-1}v$.
\end{proposition}

A similar observation as in the subsequent lemma can be found in 
\cite[Example 9]{chazal2017}.

\begin{lemma}
\label{lemma:curvature} 
Let $M\subseteq \R^d$ be a $\lip$ submanifold. Then
\[\vartheta(\xi,v)\le\varrho(\xi,v)\]
for every $(\xi,v)\in \nu_1(M)$.
\end{lemma}

\begin{proof} 
Let $(\xi,v)\in \nu_1(M)$. In the case $\varrho(\xi,v)=\infty$ 
there is nothing to show. In the  case $\varrho(\xi,v)<\infty$  assume 
instead $\vartheta(\xi,v)>\varrho(\xi,v)$.
Choose 
$r_1\in \left(\varrho(\xi,v) ,\vartheta(\xi,v)\right)$. 
Then $M\cap B_{r_1}(\xi+r_1v)\cap B_\tau(\xi)\ne\emptyset$ 
for all $\tau\in (0,\infty)$.
In particular, $M\cap B_{r_1}(\xi+r_1v)\ne\emptyset$which implies $d(\xi+r_1v,M)< r_1$.
On the other hand, by \autoref{right_projection_lemma} we have
$\xi+r_1 v\in \cE(M)$ and $p(\xi+r_1 v)=\xi$, and in particular,
$d(\xi+r_1v,M)=r_1$. This is a contradiction.
\end{proof}

\begin{lemma}\label{lemma:diffeomorphism}
Let $M$ be a $C^2$ submanifold of $\R^d$.
Then the map $F^*\colon \nu^*(M)\to \cE(M)$ is a diffeomorphism.
\end{lemma}

\begin{proof}
We have already established in \autoref{right_projection_lemma} that
$F^*$ is a homeomorphism. Furthermore,  $F^*$ is
differentiable (see \autoref{remark:F-Ck-1}).
Thus, it remains to show that the Jacobian of $F^*$ has full rank in every
point $(\xi,v)\in \nu^*(M)$. By \autoref{lemma:arnold} this could only
fail if $F(\xi,v)$ was a center of curvature of $M$.
By \autoref{lemma:curvature} and \autoref{theorem:curvature} however, no center of curvature is contained in 
$\cE(M)$.
\end{proof}

\setcounter{maintheorem}{0}
\begin{maintheorem} \label{theorem:theta-continuous}
\varthetacontinuous
\end{maintheorem}

\begin{proof}
1. $\vartheta$ is lower semi-continuous:
Let $(\xi,v)\in \nu_1(M)$.
We first consider the case where $\vartheta(\xi,v)<\infty$.
Let $\varepsilon\in (0,\infty)$ and  $r\in(\vartheta(\xi,v)-\varepsilon,\vartheta(\xi,v)) $ such that
${\big]\xi,\xi+rv\big[}\subseteq\cE(M)$ and, since $\cE(M)$ is open, there  
exists 
$\rho\in(0,r-\vartheta(\xi,v)+\varepsilon)$  with 
$B_\rho(\xi+rv)=\{z\in \R^d\colon \|z-(\xi+rv)\|<\rho\}\subseteq \cE(M)$.

Since $p$ is continuous and $p(\xi+rv)=\xi$ by \autoref{right_projection_lemma}, 
we can choose $\delta\in(0,\rho)$ such that 
\[
\forall z\in \cE(M)\colon
\|\xi+r v -z\|<\delta\Rightarrow 
\Big\|\Big(p(z)+r\tfrac{z-p(z)}{\left\|z-p(z)\right\|}\Big)-(\xi+rv)\Big\|<\rho\,, 
\]
and in particular $\forall z\in \cE(M)\colon
\|\xi+r v -z\|<\delta\Rightarrow
p(z)+r\tfrac{z-p(z)}{\left\|z-p(z)\right\|}\in \cE(M)$.
By the continuity of the endpoint map there exist $\delta_1,\delta_2\in (0,\infty)$
such that for all $(\zeta,w)\in \nu_1(M)$ it holds
\[
\|\zeta-\xi\|<\delta_1 \text{ and } \|w-v\|<\delta_2 \Rightarrow \|\xi+r v -(\zeta+r w)\|<\delta
\]
and
therefore \[
\|\zeta-\xi\|<\delta_1 \text{ and } \|w-v\|<\delta_2 \Rightarrow 
\vartheta(\zeta,w)\ge r>\vartheta(\xi,v)-\varepsilon\,.
\]  
This shows that $\vartheta$ is
lower semi-continuous in $(\xi,v)$.

The proof for the case $\vartheta(\xi,v)=\infty$ is similar and is left to 
the reader.
\smallskip

\noindent 2. $\vartheta$ is 
upper semi-continuous if $M$ is $C^2$.
Assume the opposite. Then there exist $(\xi,v) \in \nu_1(M)$,
$\alpha\in (0,\infty)$,  and a sequence $(\xi_k,v_k)_{k\in \N}$ 
in $\nu_1(M)$ converging to  $(\xi,v)$ with  $\vartheta(\xi_k,v_k)\geq \alpha$ for every $k\in \N$ but
$\vartheta(\xi,v)<\alpha$.
Choose a sequence
$(r_k)_{k\in \N}$ in $[\vartheta(\xi,v),\alpha)$ with $\lim_k r_k=\alpha$. 
From  \autoref{right_projection_lemma} it follows that 
$\xi_k+r_k v_k\in
\cE(M)$ and $p(\xi_k+r_k v_k)=\xi_k$ for every $k\in \N$. 
In other words, for every $k\in \N$ we have
$B_r\left(\xi_k+r_k v_k\right)\cap M=\emptyset$ and 
$\bar{B}_r\left(\xi_k+r_k v_k\right)\cap M=\{\xi_k\}$. 
Thus, we obtain
$B_\alpha\left(\xi+\alpha v\right)\cap
M=\emptyset$ and $\bar B_\alpha\left(\xi+\alpha v\right)\cap
M\supseteq\{\xi\}$. This together with \autoref{lemma:line-segment} 
implies that for all $r\in  [\vartheta(\xi,v),\alpha)$, 
we have that $\xi$ is the unique nearest point to 
$\xi+r v$ in $\overline{M}$. In particular, for 
$z:=\xi+\vartheta(\xi,v)v$ we have  $\{z\}\in \unpp(\overline{M})$.

Since $z\notin \cE(M)$ 
there is a sequence $(u_k)_{k\in \N}$ converging to $z$ such that 
for every $k\in \N$ we have either that $u_k$ has no nearest point on
$M$ or has at least 2 nearest points on $M$.
Consider now any sequence $(\zeta_k)_{k\in \N}$ with 
$d(u_k,\overline M)=\|u_k-\zeta_k\|$.
It is easy to see that $(\zeta_k)$ is bounded, and therefore has a 
convergent subsequence $(\zeta_{k_j})_{j\in \N}$.
But then $\|z-\lim_{j}\zeta_{k_j}\|=\lim_{j}\|u_{k_j}-\zeta_{k_j}\|\le \lim_{j}\|u_{k_j}-\xi\|=\|z-\xi\|$, such that 
$\|z-\lim_{j}\zeta_{k_j}\|=\|z-\xi\|$ and therefore  $\lim_{j}\zeta_{k_j}=\xi$.

Denote now $P_k:=\big\{\zeta\in \overline M\colon 
\|u_k-\zeta\|=d(u_k,\overline M)\big\}$. Then it holds
\begin{equation}\label{eq:hpuniform}
\forall \varepsilon>0\, \exists N_\varepsilon: \forall k\geq
N_\varepsilon:\sup_{\zeta\in P_k}\|\zeta-\xi\|\leq \varepsilon\,,
\end{equation}
since otherwise one could find a sequence $(\zeta_k)_{k\in\N}$ with 
$\zeta_k\in P_k$ having an accumulation point different from $\xi$, which 
we found to be impossible in the preceding paragraph.

It is readily checked that, since $M$ is a submanifold, 
there exists $\varepsilon\in (0,\infty)$ such that 
$\bar B_\varepsilon(\xi)\cap M=\bar B_\varepsilon(\xi)\cap \overline M$.
From Formula \eqref{eq:hpuniform} it now follows 
that there exists $N_\varepsilon$ such that for all 
$k\ge N_\varepsilon$ we have $P_k\subseteq M$. In addition, it follows
that for every $k\ge N_\varepsilon$ the point $u_k$ has at least 2 nearest
points on $M$.

We have so far succeeded, under the assumption that $\vartheta$ is not
upper semi-continuous, to show existence of $z\notin \cE(M)$ and 
of a sequence
$(u_k)_{k\in \N}$, such that for every $k\in \N$ there exist  
$\zeta_k,\eta_k\in M$ with $\zeta_k\ne \eta_k$ and 
$\|u_k-\zeta_k\|=d(u_k,M)=\|u_k-\eta_k\|$.

This means that the endpoint map $F$ is not injective on any open neighborhood of $(\xi,\vartheta(\xi,v)v)$
in $\nu(M)$. 
It follows from the inverse function theorem that the derivative of $F$ is singular at
$(\xi,\vartheta(\xi,v)v)$. 
By \autoref{lemma:arnold}, $F(\xi,\vartheta(\xi,v)v)$ is a center of 
curvature in $\xi$ in direction of $v$.
By \autoref{theorem:curvature} and \autoref{lemma:curvature} we get $\vartheta(\xi,v)=\varrho(\xi,v)$.
Since $M$ is $C^2$, the function 
$\varrho\colon \nu_1(M)\to (0,\infty]$ is continuous. We have already
shown the existence of
a sequence $(\xi_k,v_k)_{k\in \N}$ 
in $\nu_1(M)$  
converging to some pair $(\xi,v) \in \nu_1(M)$ such that 
$\vartheta(\xi_k,v_k)\geq \alpha$ for every $k\in \N$ but
$\vartheta(\xi,v)<\alpha$. From \autoref{lemma:curvature} it follows
that $\varrho(\xi,v)=\lim_{k}\varrho(\xi_k,v_k)\ge 
\lim_{k}\vartheta(\xi_k,v_k)\ge\alpha > \vartheta(\xi,v)$.

This is the desired contradiction.
\end{proof}

\begin{remark}
The proof of the first assertion of the preceeding theorem may at first sight
seem more complicated than necessary. The technicalities arise since, for example,
one cannot simply conclude $\vartheta(\zeta,w)\ge \beta$ from $B_\beta(\zeta+\beta w)\cap M=\emptyset$.
\end{remark}

\begin{example} \label{ex:theta-discont}
We construct an example of a 1-dimensional submanifold $M$ of
the $\R^2$ which is $\lip$ but for which $\vartheta$ is not continuous.
$M$ is defined as the graph of a function $f:\R\to\R$ with
$f(x)=\int_0^x g(y)dy$ and $g:\R\to\R$ is defined as
\[
g(x):=\left\{\begin{array}{lll}
0& x\le 0 \text{ or }x>1\\
x-1&2/3<x\le 1\\
3^{-(2k+1)}-x&2\cdot 3^{-2(k+1)}<x\le 2\cdot 3^{-(2k+1)}& (\forall k\in \N\cup\{0\})\\
x-3^{-2k}&2\cdot 3^{-(2k+1)}<x\le 2\cdot 3^{-2k}& (\forall k\in \N\cup\{0\}).
\end{array}\right.
\]
The function $g$ is clearly Lipschitz with Lipschitz constant equal to 1 and
$T_{(x,f(x))}=\{a (1,g(x))\colon a\in \R\}$, so $M$ is $\lip$. It is readily checked that $|g(x)|\le \frac{|x|}{2}$, so
$f(x)\leq\frac{x^2}{4}$. Since 
$f(x)=0$ for $x<0$ we have $\vartheta((x,0),(0,1))\ge 2$ for all $x<0$. 

On the intervals $[2\cdot 3^{-(2k+1)}, 2\cdot 3^{-2k}]$ we have
$f(x)=(x-3^{-2k})^2+f(3^{-2k})$, such that using also \autoref{lemma:curvature}
\[
\vartheta\big((3^{-2k},f(3^{-2k}),(0,1)\big)\le\varrho\big((3^{-2k},f(3^{-2k}),(0,1)\big)=\frac{1}{2}\,.
\]
So we have found a sequence $(\xi_k,v_k)\in \nu_1(M)$ with 
$\lim_{k\to\infty}(\xi_k,v_k)=((0,0),(0,1))$ and
$\vartheta(\xi_k,v_k) \le \frac{1}{2}<2\leq\lim_{x\nearrow 0}\vartheta\big((x,0),(0,1)\big)$. \\

For illustration we plot the graphs of $f'$ and $f$:\\

\includegraphics[scale=0.2]{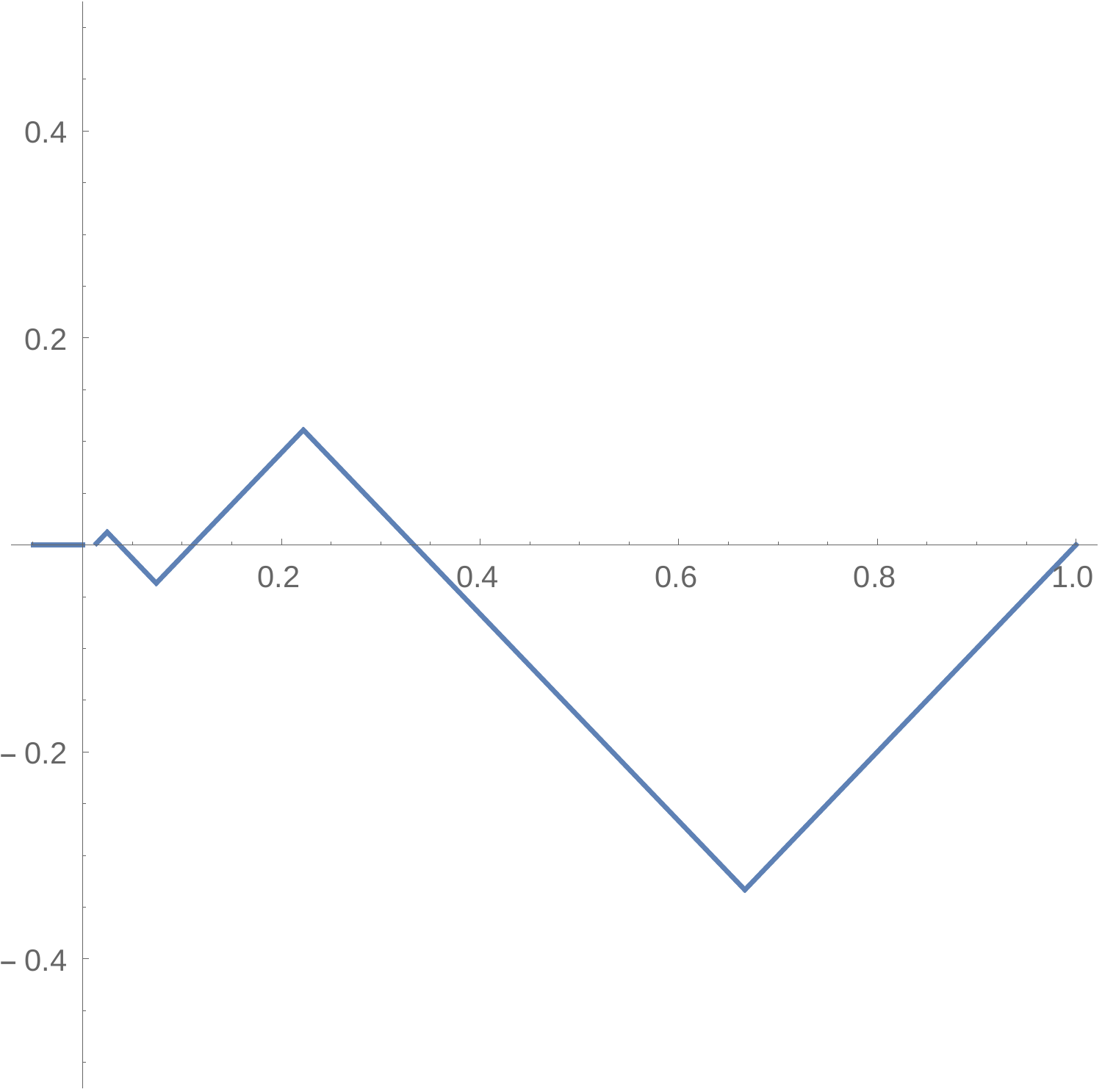}
\includegraphics[scale=0.2]{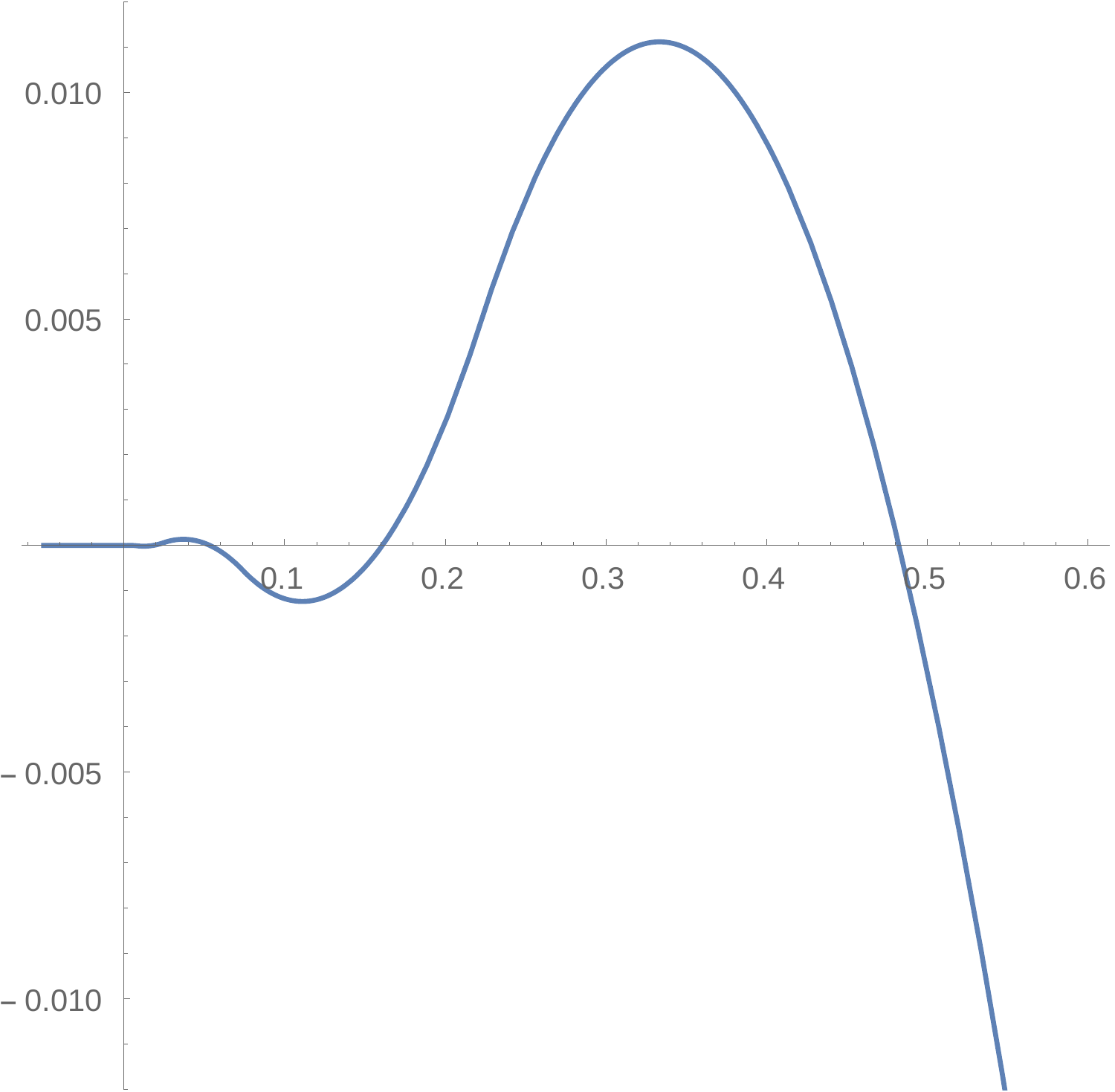}
\end{example}
\medskip

\begin{example}
Let $M=\big([1,\infty)\times\{0\}\big)\cup\{(x_1,x_2)\in \R^2\colon (x_1-1)^2+(x_2-1)^2=1,x_1<1,x_2<1\}\cup \big(\{0\}\times[1,\infty)\big)$.

Obviously, $M$ is not $C^2$, but it is readily checked that $M$ is $\lip$ 
and that $\vartheta$ is continuous.
\end{example}

\begin{example}
In general, continuity of $\vartheta$ is all we get, even if $M$ is $C^\infty$:
it is easy to construct examples of $C^\infty$-submanifolds $M$ such that
$\vartheta$ is not differentiable. We provide a drawing:
\begin{center}
\begin{picture}(0,0)%
\includegraphics{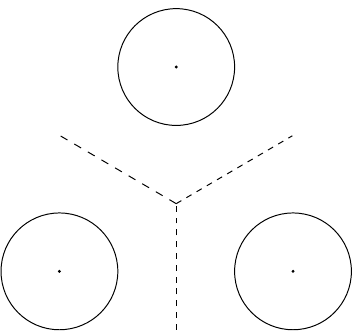}%
\end{picture}%
\setlength{\unitlength}{4144sp}%
\begingroup\makeatletter\ifx\SetFigFont\undefined%
\gdef\SetFigFont#1#2#3#4#5{%
  \reset@font\fontsize{#1}{#2pt}%
  \fontfamily{#3}\fontseries{#4}\fontshape{#5}%
  \selectfont}%
\fi\endgroup%
\begin{picture}(2686,2519)(893,-2070)
\put(2661,314){\makebox(0,0)[lb]{\smash{{\SetFigFont{10}{12.0}{\rmdefault}{\mddefault}{\updefault}{\color[rgb]{0,0,0}$k_1$}%
}}}}
\put(1806,-1361){\makebox(0,0)[lb]{\smash{{\SetFigFont{10}{12.0}{\rmdefault}{\mddefault}{\updefault}{\color[rgb]{0,0,0}$k_2$}%
}}}}
\put(3105,-110){\makebox(0,0)[lb]{\smash{{\SetFigFont{10}{12.0}{\rmdefault}{\mddefault}{\updefault}{\color[rgb]{0,0,0}$M=k_1\cup k_2\cup k_3$}%
}}}}
\put(3516,-1263){\makebox(0,0)[lb]{\smash{{\SetFigFont{10}{12.0}{\rmdefault}{\mddefault}{\updefault}{\color[rgb]{0,0,0}$k_3$}%
}}}}
\end{picture}%

\end{center}
\end{example}

We give two corollaries of \autoref{theorem:theta-continuous}, the second stating that $\cE(M)$ is a vector bundle.

\begin{corollary}\label{cor:scalingcontinuous}
Let $M$ be a $C^2$ submanifold of $\R^d$. Then the maps 
\begin{align*}
\overline{\vartheta}\colon\nu^*(M)&\to [1,\infty)\\
(\xi,v)&\mapsto \begin{cases}\frac{\vartheta\left(\xi,\frac{v}{\|v\|}\right)}{\vartheta\left(\xi,\frac{v}{\|v\|}\right)-\|v\|}, & \text{if }\vartheta\left(\xi,\frac{v}{\|v\|}\right)<\infty \text{ and } v\neq 0,\\
1 & \text{otherwise},\end{cases}
\end{align*}
\begin{align*}
\underline{\vartheta}\colon\nu(M)&\to (0,1]\\
(\xi,w)&\mapsto \begin{cases}\frac{\vartheta\left(\xi,\frac{w}{\|w\|}\right)}{\vartheta\left(\xi,\frac{w}{\|w\|}\right)+\|w\|}, & \text{if }\vartheta\left(\xi,\frac{w}{\|w\|}\right)<\infty \text{ and } w\neq 0,\\
1 & \text{otherwise},\end{cases}
\end{align*}
are continuous.
\end{corollary}

\begin{corollary}\label{lemma:fiberbundle}
Let $M$ be a $C^2$ submanifold of $\R^d$.
Then the map
\begin{align*}
\varphi\colon\nu^*(M)&\to \nu(M)\\
(\xi,v)&\mapsto (\xi,\overline{\vartheta}(\xi,v)\,v)
\end{align*}
is a homeomorphism with inverse
\begin{align*}
\varphi^{-1}\colon\nu(M)&\to \nu^*(M)\\
(\xi,w)&\mapsto (\xi,\underline{\vartheta}(\xi,w)\,w).
\end{align*}
Moreover, the homeomorphism $\sigma:=\varphi\circ (F^*)^{-1}\colon\cE(M)\to\nu(M)$ makes $\cE(M)$ a vector bundle with bundle map
$p\colon \cE(M)\to M$. Through transport from $\nu(M)$, the vector operations
on $p^{-1}(\{\xi\})$ for $\xi\in M$ are given by $+_\xi=\sigma^{-1}\circ+\circ
(\sigma\times\sigma)$ and $\cdot_\xi=\sigma^{-1}\circ\cdot\circ
(\mathrm{id}_\R\times\sigma)$, where $(+,\cdot)$ are the vector operations on the
normal space $(T_\xi M)^\perp$.  
\end{corollary}

Consider the following concept which has first been defined in 
\cite{federer1957}.

\begin{definition}
Let $M\subseteq \R^d$. 
The {\em reach} of $M$ is the largest $\varepsilon_0$ (if it exists)
such that $M^{\varepsilon_0}\subseteq \unpp(M)$, i.e.,
\[
\reach(M)=\sup\{\varepsilon\in (0,\infty)\colon M^\varepsilon\subseteq \unpp(M)\}\,.
\]
Note that $\reach(M)\in [0,\infty]$.
We call $M$ a set of {\em positive reach} iff $\reach(M)>0$. 
\end{definition}

The following generalizes \cite[Lemma]{Foote1984}, 
which states that 
compact $C^2$ submanifolds have positive reach.

\begin{proposition}\label{theorem:reach}
Let $M$ be a $\lip$ submanifold. 
Then 
$\reach(M)\le \inf\{\vartheta(\xi,v)\colon (\xi,v)\in \nu_1(M)\}$,
with equality if $M$ is a closed subset of $\R^d$.

In particular, if $M$ is compact, then 
$\reach(M)=\min\{\vartheta(\xi,v)\colon (\xi,v)\in \nu_1(M)\}>0$.
\end{proposition}

\begin{proof}
%
%
If $\reach(M)=0$, then trivially $\reach(M)\le \inf\{\vartheta(\xi,v)\colon (\xi,v)\in \nu_1(M)\}$.
If $\reach(M)>0$, let $\varepsilon \in \big(0,\reach(M)\big)$. Then $M^\varepsilon$ is open and 
$M^\varepsilon \subseteq \unpp(M)$. Thus $M^\varepsilon\subseteq \cE(M)$. 
So for every $(\xi,v)\in \nu_1(M)$ and every $r\in (0,\varepsilon)$ we have
$\xi+rv\in\cE(M)$ and thus $r<\vartheta(\xi,v)$, so $\varepsilon \le \vartheta(\xi,v)$.
From this it follows that $\reach(M)\le \vartheta(\xi,v)$ for all
$(\xi,v)\in \nu_1(M)$.

\medskip

Suppose now that $M$ is a closed subset of $\R^d$.
If $\inf\{\vartheta(\xi,v)\colon (\xi,v)\in \nu_1(M)\}=0$ there is nothing to show. Otherwise 
let $\varepsilon\in(0,\infty)$ with   $\varepsilon<\inf\{\vartheta(\xi,v)\colon (\xi,v)\in \nu_1(M)\}$ and let $x\in M^\varepsilon\setminus M$. 
Let $\xi\in M$ be a nearest point to $x$ on $M$.
Then $(\xi,x-\xi)\in \nu(M)$ by \autoref{remark:first-order}. We have
$\vartheta(\xi,\tfrac{x-\xi}{\|x-\xi\|})>\varepsilon$,
so $\xi+r \tfrac{x-\xi}{\|x-\xi\|}\in \cE(M)$ for all $r\in (0,\varepsilon)$.
In particular, $x\in \cE(M)$ and $\xi=p(x)$. 
So $M^\varepsilon\subseteq \unpp(M)$, and therefore $\reach(M)\ge \varepsilon$.

\medskip

Now assume $M$ is compact. Then also $\nu_1(M)$ 
is compact since it is homeomorphic to the product of compact spaces. By \autoref{theorem:theta-continuous}, $\vartheta$ is 
lower semi-continuous on $\nu_1(M)$, and therefore attains its minimum
in some point $(\xi_0,v_0)\in \nu_1(M)$.
By \autoref{theorem:dudek-holly}, $\vartheta(\xi_0,v_0)>0$.
\end{proof}

\begin{remark}
Let $M\subseteq \R^d$. 
If $M$ has positive reach, then $M$ is closed:
suppose $M$ were not closed. Let $z \in \bar  M \setminus M$. 
Then $z$ has no nearest point on $M$, and $z\in M^\varepsilon$ 
for all $\varepsilon\in (0,\infty)$. Thus, $M^\varepsilon \notin \unpp(M)$
for every $\varepsilon\in (0,\infty)$, so that $\reach(M)=0$.
\end{remark}

\section{Derivatives of $p$ and $\delta_M$}\label{sec:regularity}

A proof of the following theorem can be found in \cite{dudekholly}.
We give a different proof here, which makes it an immediate 
consequence of 
\autoref{right_projection_lemma} and \autoref{lemma:diffeomorphism}.

\begin{theorem}\label{theorem:diffproj}
\pdifferentiable{}
\end{theorem}

\begin{proof}
If $k=1$ then the claim is only that $p$ is continuous and 
$\delta_M$ is differentiable.
But continuity of $p$ is the content of \autoref{lemma:continuous},
differentiability of $\delta_M$ is settled by \citet[Theorem 2]{Foote1984}.

Consider now  the case $k\ge 2$.
We show that $p$ is differentiable.
We have already shown in \autoref{lemma:diffeomorphism} that 
$F^*\colon \nu^*(M)\to \cE(M)$
is a diffeomorphism, and by \autoref{right_projection_lemma},
$p(F^*(\xi,v))=\xi$ for all $(\xi,v)\in \nu^*(M)$.
We have $(F^*)^{-1}(x)=(p(x),x-p(x))$ for all $x\in \cE(M)$,
such that 
\begin{align*}
p(x)=F^*(p(x),0)=F^*(\tilde p(p(x),x-p(x)))=F^*(\tilde p((F^*)^{-1}(x)))
\end{align*}
for all $x\in \cE(M)$, where $\tilde p\colon \nu(M)\to \nu(M)$ is the
projection onto $M\times \{0\}\subseteq\nu(M)$.

We obtain $p=F^*\circ \tilde p\circ (F^*)^{-1}$.
The function $F^*$ is $C^{k-1}$ and so is therefore also $(F^*)^{-1}$.
The projection $\tilde p$ is clearly $C^{k-1}$, and so is $p$.\\

For the regularity of $\delta_M$ we use the argument given 
in \cite{Foote1984}, which we repeat here for convenience of the reader: 
$\delta_M(x)=\|x-p(x)\|$ for all $x\in \cE(M)$.
Then for $x\in \cE(M)\setminus M$,
\[
(D(\delta_M^2))(x) v=2(x-p(x))^\top (v-(Dp(x)) v)\,.
\]
Since $p\colon \cE(M)\to M$ and thus $Dp$ is a mapping 
between the tangent bundles, i.e., $Dp:T(\cE(M))\to T(M)$,
$Dp(x)v\in T_{p(x)}(M)$, so that $(x-p(x))^\top Dp(x) v=0$.
Hence 
$(D \delta_M^2)(x)=2(x-p(x))^\top$, which is $C^{k-1}$, since this holds for
$p$.  Thus $\delta_M$ is $C^k$.
\end{proof}

\begin{remark}
In contrast to $p$, the distance function is defined 
on the whole of $\R^d$.
Moreover, $\delta_M$ is continuous on $\R^d$.
However, it is easy to find examples of $C^\infty$-submanifolds so 
that $\delta_M$ is not differentiable on $\R^d\setminus M$,
for example, if $M=\{(x,y)\in \R^2\colon y=x^2 \}$, then
$\delta_M$ is not differentiable in the points 
$\{(0,y)\colon y> 1/2\}$.
\end{remark}

\citet[Lemma 4'.1]{dudekholly} compute the (Frechet-)derivative of $p$ .  In the following theorem we give a slightly different form which
is more suitable for showing the subsequent 
\autoref{lemma:Dp-bounded} and \autoref{corollary4}.
The proof uses a method different from that in \cite{dudekholly}.

\setcounter{maintheorem}{2}
\begin{maintheorem}\label{theorem:Dp-formula}
\theoremDpformula{}
\end{maintheorem}

We see that, in general, $Dp(x)$ may explode as $x$ approaches the boundary of
$\cE(M)$, even for $M$ that are well behaved, like a circle in the plane.
The next two results state that the higher derivatives of $p$ on  $M^\varepsilon\setminus M$ with $\varepsilon<\reach(M)$ are bounded, provided the higher-reivatives of normal vectors
are bounded.

\begin{corollary}\label{lemma:Dp-bounded}
Let 
$M\subseteq \R^d$ be a $C^k$ submanifold with $k\ge 2$, 
let $\reach(M)>0$ and $\varepsilon\in (0,\reach(M))$. 
Suppose there exists a constant $K\in (0,\infty)$ and 
a family of $C^{k-1}$ moving frames $(V_i,n_{i,m+1},\dots,n_{i,d})_{i\in I}$ 
such that $M=\bigcup_{i\in I} V_i$ and 
$\|D^jn_{i,\ell}\|$ is bounded by $K$ for every $i\in I$, 
$\ell\in\{m+1,\dots,d\}$ 
and $j\in\{2,\dots,k-1\}$.
Then $D^jp$ is bounded on  $M^{\varepsilon}$, for every 
$j\in \{2,\dots,k-1\}$.
\end{corollary}

\begin{proof}
Let $x\in M^\varepsilon$, i.e., $\|x-p(x)\|<\varepsilon$.
Let $\lambda_1,\dots,\lambda_m$ be the eigenvalues of $L_{\xi,v}$.
Since $L_{\xi,v}$ is self-adjoint by \autoref{theorem:curvature},  
$\|L_{\xi,v}\|=\max(|\lambda_1|,\dots,|\lambda_m|)$, and there is no loss 
of generality in assuming 
$\|L_{\xi,v}\|=|\lambda_1|$. 

If $\lambda_1\ge 0$, then 
by \autoref{theorem:curvature}, 
$\|L_{\xi,v}\|=\varrho(\xi,v)^{-1}$.  
We have $\|x-p(x)\|<\varepsilon<\varepsilon_0\le \vartheta(\xi,v)\le\varrho(\xi,v)$, by 
$x\in M^\varepsilon$, \autoref{theorem:reach}, and 
\autoref{lemma:curvature}. Therefore,   
\[
\big\|\|x-p(x)\|L_{p(x),v}\big\|
=
\|x-p(x)\|\left\|L_{p(x),v}\right\| 
<\varepsilon\varrho(\xi,v)^{-1}
\le \frac{\varepsilon}{\varepsilon_0}<1\,.
\]
If $\lambda_1<0$, then 
$\|L_{\xi,v}\|=\|L_{\xi,-v}\|=-\lambda_1=\varrho(\xi,-v)$.  
From this we get 
 $\|x-p(x)\|<\varepsilon<\varepsilon_0\le \vartheta(\xi,-v)\le\varrho(\xi,-v)$, 
and thus again 
\[
\big\|\|x-p(x)\|L_{p(x),v}\big\|
=
\|x-p(x)\|\left\|L_{p(x),-v}\right\| 
<\varepsilon\varrho(\xi,-v)^{-1}
\le \frac{\varepsilon}{\varepsilon_0}<1\,.
\]
 
Therefore 
$\id_{T_{p(x)}(M)}-\|x-p(x)\|L_{p(x),v}$
with $v=\tfrac{x-p(x)}{\|x-p(x)\|}$
is invertible,
and 
\begin{align*}
\Big\|\big(\id_{T_{p(x)}(M)}-\|x-p(x)\|L_{p(x),v}\big)^{-1}\Big\|
&=\Big\|\sum_{\ell=0}^\infty \big(\|x-p(x)\|L_{p(x),v})^\ell\Big\|\\
&\le\sum_{\ell=0}^\infty \big\|\|x-p(x)\|L_{p(x),v}\big\|^\ell\\
&\le\sum_{\ell=0}^\infty \left(\frac{\varepsilon}{\varepsilon_0}\right)^\ell
<\big(1-\frac{\varepsilon}{\varepsilon_0}\big)^{-1}\,.
\end{align*}

Let now $(V,n_{m+1},\dots,n_d)$ be a moving frame with $x\in V$ and 
derivatives bounded by $K$.
Write down  equation \eqref{eqn:Dp}  
and write $n$ for the matrix $(n_{m+1},\dots,n_d)$: 
\begin{align*}
\id_{\R^d}-n(p(x))n(p(x))^\top&=J(x)Dp(x)\,,
\end{align*}
where $J(x):=\Big(\id_{T_{p(x)}(M)}+\sum_j\big\langle x-p(x), n_j(p(x))\big\rangle 
(D n_j)(p(x))\Big)$.
From this and the fact that 
$J(x)$ is invertible with uniformly bounded derivative,
we get 
\begin{align}\label{eqn:Dp2}
Dp(x)=
J(x)^{-1}\big(\id_{\R^d}-n(p(x))n(p(x))^\top\big)
\end{align}
and thus $Dp$ is uniformly bounded.
Differentiating the right-hand side of \eqref{eqn:Dp2} 
involves sums of products of $n$, $Dn$, $D^2n$, $p$, $Dp$, 
their rows, columns and transposes, and  
$J(x)^{-1}$, which are all bounded.
From this we get boundedness of $D^2p$.
By induction we get boundedness of $D^jp$, $j\in\{1,\dots,k-1\}$ from
boundedness of $D^jn$, $j\in\{1,\dots,k-1\}$. 
\end{proof}

The situation simplifies if $M$ is hypersurface, i.e., a $(d-1)$-dimensional 
submanifold of $\R^d$. Then  
on every connected open subset $V\subseteq M$ there exist at most 
two $C^1$ functions $n$ with the property that $(V,n)$ is a unit normal
field, and those functions differ only by their sign. If $M$ is not orientable, 
then (per definition) it is not possible to find a unit normal vector $(V,n)$ with $V=M$. Nevertheless, every $(V,n)$ is automatically $C^{k-1}$ and $\sup_{\xi\in M}\|D^jn(\xi)\|$ does not depend on the choice of a particular unit normal vector.
With this, the following corollary is an immediate consequence of 
 \autoref{lemma:Dp-bounded}.

\begin{corollary}\label{corollary4}
Let $M\subseteq \R^d$ be a $C^k$ hypersurface with $k\ge 2$.
Moreover, let $\reach(M)>0$ and $\varepsilon\in (0,\reach(M))$. 
If there exists $K\in (0,\infty)$ such that $\sup_{\xi\in M}\|D^jn\|\le K$ for every $j\in\{2,\dots,k-1\}$,
then $D^jp$ is bounded on  $M^{\varepsilon}$ for every 
$j\in \{2,\dots,k-1\}$.
\end{corollary}

\section{The converse to \autoref{theorem:dudek-holly}}\label{section:converse}

In their article \citep{dudekholly} from 1994, Dudek and Holly prove that each
point of a $\lip$ submanifold of $\R^d$ (see \autoref{def:lip}) with
dimension different from $d$ possesses a neighborhood in the ambient space
$\R^d$ which has the $\unpp$. This was the assertion of Theorem
\ref{theorem:dudek-holly}, or \cite[Theorem 3.8]{dudekholly} in the original
paper. In \autoref{theorem:pr-lip} in this section we show the converse to
the theorem, starting from topological submanifolds: 
if a topological submanifold $M$ of the $\R^d$ with dimension $m\neq d$ is such that each point $\xi$ of $M$ has an $\R^d$-neighborhood $U(\xi)$ such that $U(\xi)\subseteq \unpp(M)$, then $M$ is $\lip$. 

\autoref{theorem:pr-lip} will be formulated and proven after two essential lemmas below. The proof's core is \autoref{lemma:wurscht}, which allows the construction of normal and tangent spaces to $M$ by merely using the property that each point of $M$ has a neighborhood admitting unique projections onto $M$. The proof of \autoref{lemma:wurscht} relies on an iterative application of the Borsuk-Ulam theorem. We shall also use the following lemma.

\begin{lemma}\label{lemma:pm1-convex}
Let $M$ be a subset of $\R^d$ and $U\subseteq \unpp(M)$.
If $U$ is convex, then for every $\xi \in M$ the set $U\cap p^{-1}(\xi)$ is convex.
\end{lemma}

\begin{proof}
Let $x_1,x_2\in U$ with $p(x_1)=p(x_2)=\xi$, and let $x_3\in {]x_1,x_2[}$.
Then it is not hard to check that $B_{\|x_3-\xi\|}(x_3)\subseteq  B_{\|x_1-\xi\|}(x_1)\cup B_{\|x_2-\xi\|}(x_2)$.
Since $\big(B_{\|x_1-\xi\|}(x_1)\cup B_{\|x_2-\xi\|}(x_2)\big)\cap M=\emptyset$,
$\big(\bar B_{\|x_1-\xi\|}(x_1)\cup \bar B_{\|x_2-\xi\|}(x_2)\big)\cap M=\{\xi\}$ we have 
$B_{\|x_3-\xi\|}(x_3)\cap M=\emptyset$ and $\bar B_{\|x_3-\xi\|}(x_3)\cap M\subseteq\{\xi\}$. On the other hand clearly 
$\xi\in \bar B_{\|x_3-\xi\|}(x_3)\cap M$, and therefore $\{\xi\}= \bar B_{\|x_3-\xi\|}(x_3)\cap M$.
\end{proof}

We remind the reader that for a set $A\subseteq \R^d$
and $r\in (0,\infty)$  we denote $A^r=\{x\in \R^d\colon d(x,A)<r\}$.
For example, if $H$ is a 2-dimensional subspace of the $\R^3$ and
$S=\bar B_1(0)\setminus B_1(0)$, $0<r_1\le r_2$,
then 
$A=\big((r_2 S)\cap H\big)^{r_1}$ is the interior of a filled torus, 
the case $r_1=r_2$ giving a ``horn torus''. 
Note that $A\cap H^\perp=\emptyset$, even in the latter case.

\begin{lemma}\label{lemma:wurscht}
Let $M$ be an $m$-dimensional topological submanifold of $\R^d$ with
the property that for every $\eta\in M$ there exists 
$U\subseteq \unpp(M)$ with $U$ open and $\eta\in U$.
Then for every $\eta\in M$ there exists $r\in (0,\infty)$ such that
for every $\xi\in M\cap B_r(\eta)$ there exists 
a $(d-m)$-dimensional subspace $N_\xi$  of
$\R^d$ with 
\begin{equation}\label{eq:wurscht}
\Big(\xi+\big((r S)\cap N_\xi\big)\Big)^r\cap M=\emptyset 
\quad\text{ and }\quad\overline{\Big(\xi+\big((r S)\cap N_\xi\big)\Big)^r}\cap M=\{\xi\}\,,
\end{equation}
where $S$ denotes the $(d-1)$-dimensional unit sphere.
\end{lemma} 

\begin{proof}
There exist open sets $V,W\subseteq \R^d$ with $0\in V$, $\eta\in W$ and a 
continuous map $\Psi\colon V\to W$ with $\Psi(0)=\eta$ and 
$M\cap W=\{\Psi(y)\colon y\in V,\,y_{m+1}=\dots=y_d=0\}$. 
There exists $r\in (0,\infty)$ such that $B_{3r}(\eta)\subseteq U\cap W$. 
Let $\xi\in M\cap B_r(\eta)$.
Denote by $\pi_{d-1}\colon \R^d\to \R^{d-1}$ the projection
defined by $\pi_{d-1}(y_1,\dots,y_d):=(y_1,\dots,y_{d-1})$.
By \autoref{lemma:continuous}, and since $\xi+r S\subseteq B_{3r}(\eta)\subseteq U\subseteq \unpp(M)$,
$f_1\colon S \to \R^{d-1}$, $f_1(v):=\pi_{d-1}\big(\Psi^{-1}(p(\xi+rv))\big)$ is continuous. By the Borsuk-Ulam
theorem there exists $n_d\in S$ with $f_1(n_d)=f_1(-n_d)$, yielding 
$p(\xi+rn_d)=p(\xi-rn_d)=:\zeta$. Obviously, $\|\xi+rn_d-\zeta\|\le r$
and $\|\xi-rn_d-\zeta\|\le r$ and  $\|\xi+rn_d-(\xi-rn_d)\|=2r$. Thus,
by the triangle inequality $\|\xi+rn_d-\zeta\|= r$
and $\|\xi-rn_d-\zeta\|= r$ and therefore $\zeta=\xi$. Let $N^{(1)}$ be the hyperplane
$\{v\in \R^d\colon \langle v,n_d\rangle=0\}=n_d^\perp$.

Now let $\pi_{d-2}\colon \R^d\to \R^{d-2}$ be the projection
defined by $\pi_{d-2}(y_1,\dots,y_d):=(y_1,\dots,y_{d-2})$, and let
$f_2\colon S\cap N^{(1)} \to \R^{d-2}$, $f_2(v):=\pi_{d-2}\big(\Psi^{-1}(p(\xi+rv))\big)$. Applying the Borsuk-Ulam theorem again yields 
$n_{d-1}\in S\cap N^{(1)}$ with 
$p(\xi+rn_{d-1})=p(\xi-rn_{d-1})=\xi$, as before.
Let $N^{(2)}$ be the $d-2$ dimensional space
$N^{(2)}=\{v\in \R^d\colon  \langle v,n_{d-1}\rangle=0\text{ and }\langle v,n_d\rangle=0\}=n_{d-1}^\perp\cap n_d^\perp$.

By iterating this procedure 
 we get $d-m$ orthonormal vectors 
$n_{m+1},\dots,n_d$ for which $p(\xi+rn_j)=p(\xi-rn_j)=\xi$, $j=m+1,\dots,d$,
and $d-m$ subspaces $N_\xi:=N^{(d-m)}\subseteq \dots \subseteq N^{(1)}$.
Note that $N_\xi$ is the linear space spanned by $n_{m+1},\dots,n_d$. 
 
Let 
\[
K:=r\conv\{n_{m+1},\dots,n_d,-n_{m+1},\dots,-n_d\}\,.
\]
By construction of the $n_j$'s and by 
\autoref{lemma:pm1-convex}, we have  $p(\xi+K)=\{\xi\}$.
Using \autoref{lemma:extend-segment} we get 
$p\Big(\xi+\big((rS)\cap N_\xi\big)\Big)=\{\xi\}$. 
From this the assertion follows.
\end{proof}

If $t\in \R^d\setminus \{0\}$ is a vector and $T\subseteq \R^d$ is a 
non-trivial linear subspace, then
$\ang(t,T):=\min\{\ang(t,t_2)\colon t_2\in T\setminus\{0\}\}$.

If  $T_1,T_2\subseteq \R^d$  are two
non-trivial linear subspaces, we define 
\[\ang(T_1,T_2):=\max\big\{\min\{\ang(t_1,t_2)\colon t_2\in T_2\setminus\{0\}\}\colon t_1\in T_1\setminus\{0\}\big\}\,.\]
Note that $d_H(T_1,T_2):=2\arcsin(\ang(T_1,T_2)/2)$ for $T_1,T_2\in G(m,\R^d)$.

\setcounter{maintheorem}{1}
\begin{maintheorem}\label{theorem:pr-lip}
\theoremprlip
\end{maintheorem} 

\begin{proof}
\noindent {\tt Step 1:} We show that 
$M$ is locally the graph of a function $\Phi$ over an $m$-dimensional subspace of $\R^d$.

Let $\eta\in M$. By \autoref{lemma:wurscht} there exists 
$r\in (0,\infty)$ such that for every $\xi\in M\cap B_r(\eta)$ there exists 
a $(d-m)$-dimensional subspace $N_\xi$  of
$\R^d$ and such that 
\eqref{eq:wurscht} is satisfied. 
Moreover, $r$ can by chosen so that $M_1:=M\cap B_r(\eta)$ is homeomorphic
to an open subset of $\R^m$.
For every $\xi\in M$ with $\|\xi-\eta\|<r$ write $T_\xi:=N_\xi^\perp$, where $N_\xi$ is the linear space constructed
in \autoref{lemma:wurscht}.
Consider the map $f:M_1\to T_\eta$, $f(\xi)=P_{T_\eta}(\xi-\eta)$. 
The map $f$ is injective: suppose to the contrary that 
there exist $\xi,\zeta\in M_1$ with
$P_{T_\eta}(\xi-\eta)=P_{T_\eta}(\zeta-\eta)$, so $\ang(\zeta-\xi,T_\eta)=\frac{\pi}{2}$. By \eqref{eq:wurscht} and  $\|\xi-\eta\|<r$
we have $\ang(\xi-\eta,T_\eta)<\frac{\pi}{6}$, and similarly
$\ang(\eta-\xi,T_\xi)<\frac{\pi}{6}$, so that 
$\ang(T_\xi,T_\eta)<\frac{\pi}{3}$.  
We also have 
$\ang(\zeta-\xi,T_\xi)<\frac{\pi}{6}$.  
But $\ang(\zeta-\xi,T_\eta)\le\ang(\zeta-\xi,T_\xi)+\ang(T_\xi,T_\eta)
<\frac{\pi}{6}+\frac{\pi}{3}<\frac{\pi}{2}$, a contradiction.

The set $V:=\{P_{T_\eta}(\xi-\eta)\colon \xi\in M_1\}$ is open in $T_\eta$
by Brouwer's invariance of domain theorem.
Thus the map $\Phi\colon V \to N_\eta$, $\Phi(t)=f^{-1}(t)-\eta-t$
provides us with a parametrization of $M_1$ via $t\mapsto \eta+t+\Phi(t)$. 
We see that $M_1$ is the graph of a function.

\noindent {\tt Step 2:} We show that $\Phi$ is differentiable.

Let $t\in V$. Then $\xi:=\eta+t+\Phi(t)\in M_1$ and $\ang(T_\xi,T_\eta)<\frac{\pi}{3}$ such that $T_\xi\cap (T_\eta^\perp)=\{0\}$. Thus, there exists a linear mapping 
$A_t:T_\eta\to T_\eta^\perp$ with $T_\xi=\{h+A_t h\colon h\in T_\eta\}$.
Since $\ang(T_\eta,T_\xi)<\frac{\pi}{3}$ and because of \eqref{eq:wurscht}
it holds 
$-\frac{8}{r}\|h\|^2\le \|\Phi(t+h)-\Phi(t)-A_t h\|\le \frac{8}{r}\|h\|^2$ for all  $h$ with sufficiently small norm. 
Thus, $\Phi$ is differentiable in $t$, with derivative $A_t$.
In particular, $T_\xi(M)=T_\xi$ for every $\xi\in M$.

\noindent {\tt Step 3:} We show that the derivative of $\Phi$ is Lipschitz
so that, in particular, $M$ is $C^1$.

So far we have succeeded in showing that for all $t\in V$ and all 
$h\in T_\eta$ satisfying $t+h\in V$ it holds 
\[
\Phi(t+h)=\Phi(t)+A_t h  +\kappa(t,h) 
\]
for some remainder function $\kappa$ satisfying 
$\|\kappa(t,h)\|<\frac{8}{r}\|h\|^2$.

Now fix $t\in V$ and let $a$ be such that $B_a(t)\cap T_\eta\subseteq V$.
Let $h,k\in T_\eta$ with $0<\|h\|=\|k\|<\frac{a}{2}$. Then 
\begin{align*}
\Phi(t+k)&=\Phi(t+h)+ A_{t+h}(k-h) +\kappa(t+h,k-h)\\
\Phi(t-k)&=\Phi(t+h)+A_{t+h}(-k-h) +\kappa(t+h,-k-h)\\
\Phi(t+k)&=\Phi(t)+ A_tk +\kappa(t,k)\\
\Phi(t-k)&=\Phi(t)+ A_t(-k) +\kappa(t,-k)\,.
\end{align*}
We add the first and fourth equation and subtract the second and third 
to get
\[
0=2 (A_{t+h}-A_t)k+\kappa(t+h,k-h)-\kappa(t+h,-k-h)-\kappa(t,k)+\kappa(t,-k)
\]
and thus
\begin{align*}
2 \|(A_{t+h}-A_t)k\|
&\le \frac{8}{r}(\|k-h\|^2+\|-k-h\|^2+\|k\|^2+\|-k\|^2)\\
&= \frac{8}{r}\big(2\|h\|^2+4\|k\|^2\big)=\frac{48}{r}\|h\|^2\,,
\end{align*}
since we assumed $\|k\|=\|h\|$. Therefore,
\begin{align*}
\|k\|^{-1}\|(A_{t+h}-A_t)k\|
&\le \frac{24}{r}\|k\|^{-1}\|h\|^2=\frac{24}{r}\|h\|\,,
\end{align*}
and, because $k$ was arbitrary with $\|k\|=\|h\|$, it follows that  
$\|A_{t+h}-A_t\|\le \frac{24}{r}\|h\|$, which means that the mapping  $t\mapsto A_t$ is Lipschitz.\\

\noindent {\tt Step 4:} We show that $M$ is $\lip$. 
For this it suffices to show that $\xi\mapsto T_\xi$ is Lipschitz on 
$M\cap B_\frac{r}{4}(\eta)$.
Let now $\xi,\zeta\in M\cap B_\frac{r}{4}(\eta)$, thus $\zeta\in M\cap B_{\frac{r}{2}}(\xi)$. 
Note that since $\xi \in M\cap B_\frac{r}{4}(\eta)$, clearly $M\cap B_{\frac{r}{2}}(\xi)\subseteq M_1$ and therefore {\tt Step 1} -- {\tt Step 3} can be performed for $\xi, M\cap B_{\frac{r}{2}}(\xi)$ in position of $\eta, M_1$, yielding also the same Lipschitz constant in {\tt Step 3}. 
 

In particular, as done for $\eta$ in {\tt Step 1}, $M\cap B_\frac{r}{2}(\xi)$ can 
be represented as a graph over $T_\xi$: there exists $\bar\Phi\colon T_\xi\cap B_\frac{r}{2}(0)\to N_\xi$ such that 
$M\cap \{\xi+(T_\xi\cap B_\frac{r}{2}(0))+(N_\xi\cap B_r(0))\} =\{\xi+t+\bar\Phi(t)\colon t\in T_\xi\cap B_\frac{r}{2}(0)\}$ and $\bar{\Phi}(0)=0$.

Hence, $\zeta=\xi+t+\bar\Phi(t)$ for some $t\in T_\xi$,
and $T_\zeta=\{t_1+D\bar\Phi(t)t_1\colon t_1\in T_\xi\}$.
Now let $s\in T_\xi$ with $\|s\|=1$ and such that 
$\ang(T_\xi,T_\zeta)=\arctan(\|D\bar\Phi(t)s\|)$ and hence
$\ang(T_\xi,T_\zeta)\le\arctan(\|D\bar\Phi(t)\|)$.
We use the estimate $ 2 \arcsin\left(\tfrac{1}{2}\arctan(x)\right)\leq x$ 
for $x\in [0,\infty)$ (which can be shown by proving that $g$ defined by 
$g(x):=x-2 \arcsin\left(\tfrac{1}{2}\arctan(x)\right)$ satisfies $g(0)=0$ 
and $g'>0$), the Lipschitz continuity of $\bar \Phi$ from {\tt Step 3}
and Pythagoras' theorem to compute the estimate
\begin{align*}
d_H(T_\xi,T_\zeta)&=2 \arcsin\left(\tfrac{1}{2}\arctan(\|D\bar\Phi(t)\|)\right)
\le \|D\bar\Phi(t)\|\\
&\le \tfrac{24}{r}\|t\|
\le \tfrac{24}{r}\|t+\bar \Phi(t)\|
= \tfrac{24}{r}\|\xi-\zeta\|\,.
\end{align*}
\end{proof}

\begin{remark}\label{remark:blaschke} 
\citet[Theorem 1]{walter} generalizes Blaschke’s
Rolling Theorem for the boundary of a compact and path-connected 
subset $P\subseteq \R^d$.
In particular, the theorem there states that there exists $r_0>0$ such that a ball of radius $r$ rolls freely inside
$P$ and $\overline{P^c}$ for all $0\le r\le r_0$ iff $\partial P$ is a 
$\lip$ hypersurface.

For hypersurfaces this free-rolling condition is equivalent to $\reach(\partial P)\ge r_0$, using the notation of the present article. 
However, the methods used there cannot be used to prove \autoref{theorem:pr-lip}, and also there is no obvious way to generalize the whole setup in  \cite{walter} to codimensions other than 1.
\end{remark}

\section{The topological skeleton}\label{sec:skeleton}

Here we highlight the relation between $\cE(M)$ and the
topological skeleton (a.k.a.~medial axis) of $M^c$. 

\begin{definition}
Let $A\subseteq \R^d$ be a subset. 
\begin{enumerate}
\item A ball $B_r(x)$ with $r\in (0,\infty)$ is called {\em maximal in $A$}, iff
\begin{enumerate}[(i)]
\item $B_r(x)\subseteq A$
\item for all $x'\in \R^d, r'\in (0,\infty)\colon
B_r(x)\subseteq B_{r'}(x')\subseteq A$, then $r'=r,x'=x$.
\end{enumerate}
\item
Define the {\em topological skeleton} by
\[
\cS(A):=\{x\in \R^d\colon \exists r\in (0,\infty) \colon B_r(x) \text{ is maximal in }A\}
\]
\end{enumerate}
\end{definition}

\begin{remark}\label{remark:maximal-ball}Note that a ball $B_r(x)$ is
maximal in $A$ iff $B_r(x)$ is maximal in $A^\circ$. Therefore,
$\cS(A)=\cS(A^\circ)$.

Note further that if $x$ has at least 2 nearest points on $A^c$, then 
$x\in \cS(A)$.
\end{remark}

The following is an adaptation of the well-known medial axis transform, which 
allows reconstruction of an open subset of $\R^d$ from its topological skeleton.
In our version the complement of a closed set (for example a closed manifold)
is recovered from its skeleton.

\begin{proposition}[Medial axis transform, recovery from skeleton]\label{theorem:medial-axis}
\theoremmedialaxis
\end{proposition}

\begin{proof}
Let $\cB=\bigcup \big\{B_r(x)\colon x\in \cS(M^c)\text{ and } r=d(x,M)\big\}$.

The inclusion 
\(
\Big(\bigcup_{H\in  \cH(M)}H^c\Big) \cup \cB \subseteq M^c
\) is obvious.

\medskip

We show 
\(M^c\subseteq \Big(\bigcup_{H\in  \cH(M)}H^c\Big) \cup \cB \).
Note that $\bigcup_{H\in  \cH(M)}H^c$ equals the complement of the convex 
hull $\operatorname{conv}(M)$ of
$M$.
Now let $x\in M^c\setminus \operatorname{conv}(M)^c$. 
Since $M$ is closed, $x$ has nearest points on $M$. 
Consider first the case that $x$ has 2 or more distinct nearest points. Then 
$B_{d(x,M)}(x)$ is maximal in $M^c$.
Thus, $x\in \cS(M^c)$ and therefore $x\in B_{d(x,M)}(x)\subseteq \cB$.

Next, consider the case that $x$ has a unique nearest point $\zeta$ on $M$,
and let 
$\alpha=\sup\{a\in [1,\infty)\colon B_{a\|x-\zeta\|}\big(\zeta+a(x-\zeta)\big)\subseteq  M^c\}$. Since the set over which the supremum is taken contains $1$,
we see that $\alpha\in [1,\infty]$.

If  $\alpha=\infty$ then 
$B_{a\|x-\zeta\|}\big(\zeta+a(x-\zeta)\big)\subseteq  M^c$ for arbitrarily large $a$, so
$x\in H^c$ for the closed half-space 
$H=\{z\in \R^d\colon \langle z-\zeta,x-\zeta\rangle\le 0\}$.  Thus
$x\in \operatorname{conv}(M)^c$, which was excluded.

If $\alpha\in [1,\infty)$ it is easy to show that the ball $B_{\alpha\|x-\zeta\|}\big(\zeta+\alpha(x-\zeta)\big)$
is maximal. So $\zeta+\alpha(x-\zeta)\in \cS(M^c)$ and 
$x\in B_{\|x-\zeta\|}(x)\subseteq B_{\alpha\|x-\zeta\|}\big(\zeta+\alpha(x-\zeta)\big)\subseteq\cB$.
\end{proof}

The next result describes the relation between $\cE(M)$ and the skeleton of $M^c$.

\begin{proposition}\label{theorem:skeleton}
Let $M\subseteq \R^d$. Then 
\[
\cE(M)^c
=\overline{\cS\big( M^c\big)} \cup \overline{\cF(M)}
=\overline{\cS\Big(\big(\overline M\big)^c\Big)} \cup \overline{\cF(M)}\,,
\]
where $\cF(M)$ is the set of points $x\in \R^d$ 
for which there is no nearest point to $x$ in $M$.
\end{proposition}

\begin{proof}
Since $\big(\overline{M})^c=\big(M^c\big)^\circ$, 
$\cS\Big(\big(\overline{M}\big)^c\Big)=\cS(M^c)$, it is enough to prove  the first
equality.

\medskip

\noindent {\tt Step 1:}  
We show $\cE(M)^c\subseteq\overline{\cS(M^c)} \cup \overline{\cF(M)}$. 

Let $x\in \cE(M)^c$. W.l.o.g.~there exists a sequence $(x_n)_{n\in\N}$ converging to
$x$ such that either all $x_n$ have no nearest point on $M$ or  all $x_n$ have
multiple nearest points on $M$.

In the first case, all $x_n$ are in $\cF(M)$ by definition. Therefore $x\in \overline{\cF(M)}$.
In the second case, all $x_n$ have multiple nearest points on $M$. Thus $x_n$ is the center of a maximal ball in $M^c$ and therefore
$x_n$ in $\cS(M^c)$. Hence $x\in \overline{\cS(M^c)}$.

\medskip

\noindent {\tt Step 2:}  
We show $\overline{\cS( M^c)} \cup \overline{\cF(M)}\subseteq\cE(M)^c$. 

First let $x\in \cF(M)$. Then $x\in \cE(M)^c$ by definition of $\cE(M)$.
Let $x\in \cS(M^c)$ and assume $x\notin \cE(M)^c$, i.e., $x\in \cE(M)$.
Therefore, $x$ has a unique nearest point $\zeta$ on $M$.
By \autoref{lemma:extend-segment} there exists $a\in (1,\infty)$ 
such that $\zeta+a (x-\zeta)$ is the center of the ball 
$B_{a\|x-\zeta\|}(\zeta+a (x-\zeta))\subseteq M^c$, which contains the ball 
$B_{\|x-\zeta\|}(x)$. Thus, $x$ is not the center of a maximal ball in $M^c$,
contradicting $x\in \cS(M^c)$.

We have shown that $\cS(M^c) \cup \cF(M)\subseteq\cE(M)^c$.  
Since $\cE(M)^c$ is closed, also 
$\overline{\cS(M^c) \cup \cF(M)}\subseteq\cE(M)^c$.
\end{proof}

\begin{remark}
\begin{enumerate}
\item The skeleton $\cS(M^c)$ 
is not automatically closed: consider as a counterexample 
$M=\{(x,|x|)\in \R^2\colon x\in \R\}$, where 
$\cS(M^c)=\big\{(0,y)\colon y\in (0,\infty)\big\}$.
\item $\cF(M)$
is not automatically closed: consider as a counterexample 
\autoref{ex:lemma_cont}, where $x\notin \cF(M)$ but all points
to the left of $x$ lie in $\cF(M)$.
\item $\cF(M)=\emptyset$ for closed $M$. 
Note further that if $M$ is countable with no accumulation points, then
$\cE(M)^c$ is the union of the boundaries of the Voronoi cells
corresponding to $M$ (see \cite[Subsection 6.2.1]{aurenhammer}).
\end{enumerate}
\end{remark}

For the main purpose of this manuscript, closedness of the skeleton has an
important consequence. 

\setcounter{maintheorem}{3}
\begin{maintheorem}\label{theorem:theta-continuous2}
\varthetacontinuoustwo
\end{maintheorem}

Note that by virtue of \autoref{theorem:theta-continuous2} and 
\autoref{lemma:fiberbundle} 
we get that $\cE(M)$ is a fiber bundle if $M$ is $\lip$ with 
closed $\cS(M^c)$.

\begin{proof}
In view of \autoref{theorem:theta-continuous} we only need to show that 
$\vartheta$ is upper semi-continuous.
For this we follow the second part of the 
proof of  \autoref{theorem:theta-continuous},
where we construct, 
under the assumption that $\vartheta$ is not
upper semi-continuous in $(\xi,v)\in \nu_1(M)$,  
a sequence
$(u_k)_{k\in \N}$ converging to $z=\xi+\vartheta(\xi,v)v$, 
such that for every $k\in \N$ there exist  
$\zeta_k,\eta_k\in M$ with $\zeta_k\ne \eta_k$ and 
$\|u_k-\zeta_k\|=d(u_k,M)=\|u_k-\eta_k\|$.

But this implies that $u_k\in \cS(M^c)$ for all $k\in \N$. Since 
$\cS(M^c)$ is closed by assumption, we also have $z\in \cS(M^c)$.
This means that $z$ is the center of a maximal ball in $M^c$.

On the other hand, since $\vartheta$ is not
upper semi-continuous in $(\xi,v)$, there exists $\alpha>\vartheta(\xi,v)$ 
such that  for all $r\in [\vartheta(\xi,v),\alpha)$
it holds that $\xi+rv\subseteq \unpp(M)$
and $\xi$ is the unique nearest point to $\xi+rv$ on $M$
(see again the proof of  \autoref{theorem:theta-continuous}).

But this contradicts the earlier finding that $z=\xi+\vartheta(\xi,v)v$
is the center of a maximal ball in $M^c$.
\end{proof}


\section*{Appendix A: Supplementary proofs}
\addcontentsline{toc}{section}{Appendix A: Supplementary proofs}

\begin{proof}[\bf Proof of \autoref{theorem:curvature}]
We know from item \ref{item:remark_submanifold2} of 
\autoref{remark:submanifold}
that $M$ can be represented as the graph
of a function on $T_\xi(M)$ in the vicinity of $\xi$, i.e., 
there exist open sets  $W\subseteq T_\xi(M)$ and $U\subseteq T_\xi(M)^\perp$ 
with $0\in W\cap U$
and a $C^2$ function $\Phi:W\to U$ 
such that $\Phi(0)=0$ and 
$M\cap (\xi+W+U)=\{\xi+t+\Phi(t)\colon t\in W\}$.
W.l.o.g.~we may identify $T_\xi(M)$ with $\R^m\times \{0\}$ and 
$T_\xi(M)^\perp$ with $\{0\}\times\R^{d-m}$
and we may assume $\Phi(0)=\xi=0$. 
Since $\Phi$ parametrizes $M$ over the tangent space in $\xi=0$, we have 
$D \Phi(0)=0$.
By \autoref{lemma:normal-vectors} 
there exists an orthonormal moving frame  $(V,n_{m+1},\dots,n_d)$ of
$\nu(V)$ with $V\subset M\cap (W+U)$ open,
and $0\in V$, and  $v=n_{m+1}(0)$.
For $k=m+1,\dots, d$ let $\Phi_k$ denote the $k$-th component of $\Phi$,
$\Phi_{k}(y)=\langle \Phi(y),n_k(0)\rangle$.
Note that $\Phi_k$ is $C^2$
and $D\Phi_k(0)=0$,
so by Taylor's theorem
\[
\Phi_k(y)=y^\top H_k y+r_k(y)
\]
with $\lim_{y\to 0} \|y\|^{-2}r_k(y)=0$, where $H_k$ is the Hessian of 
$\Phi_k$ in $0$.

Let $\partial_j$ denote differentiation w.r.t.~to the $j$-th coordinate,
and let $e_1,\dots,e_m$ denote the canonical basis vectors of $\R^m$.
For all $y\in W$ with $y+\Phi(y)\in V$ we have that 
$\{\partial_j(y+\Phi(y))\colon j=1,\dots,m\}=\{(e_j+\partial_j \Phi(y))\colon j=1,\dots,m\}$ 
forms a basis of the tangent space of $M$ in $y+\Phi(y)$. In particular,
\begin{align*}
0&=\left\langle \partial_j(y+\Phi(y)),n_k(y+\Phi(y))\right\rangle & \forall j=1,\dots,m,\\&&k=m+1,\dots,d
\\
\Rightarrow 
0
&=\left\langle \partial_i\partial_j\Phi(0),n_k(0)\right\rangle 
+\left\langle e_j,Dn_k(0)e_i\right\rangle 
& \forall i,j=1,\dots,m,\\&& k=m+1,\dots,d
\end{align*}
so that 
\[
\left\langle e_j,-Dn_k(0)e_i\right\rangle
=\left\langle \partial_i\partial_j\Phi(0),n_k(0)\right\rangle=\partial_i\partial_j\Phi_k(0)\,,
\]
and therefore,
\[
\left\langle e_j,L_{0,n_k}e_i\right\rangle=
\left\langle e_j,-P_{T_0(M)}Dn_k(0)e_i\right\rangle
=\partial_i\partial_j\Phi_k(0)=(H_k)_{ij}\,.
\]

Thus, the Hessian of $\Phi_k$ is the matrix representation of 
the shape operator $L_{0,n_k}$ with respect to the basis $(e_1,\dots,e_m)$
of $T_0(M)$. So far we have shown that 
\begin{align*}
\Phi(y)=\sum_{k=m+1}^d \langle y, L_{0,n_k} y  \rangle n_k(0) + r(y)
\end{align*}
with $\lim_{y\to 0}\|y\|^{-2}r(y)=0$ ($r=r_{m+1}n_{m+1}(0)+\dots+r_d n_d(0)$).
In particular, the Hessian is self-adjoint and so
is $L_{0,n_k}$, for every $k\in\{m+1,\dots,d\}$.

Let $\lambda_1\ge \dots\ge \lambda_m$ be the eigenvalues of $H_{m+1}$ 
(and thus of $L_{0,n_{m+1}}$).
There is no loss of generality in assuming that $(e_1,\dots,e_m)$
are the corresponding eigenvectors of $H_{m+1}$  and thus  of $L_{0,n_{m+1}}$.

\medskip
\noindent{\tt Case 1:} $\lambda_1< 0$. Then for all $R>0$ 
and for all $z\in B_R(Rn_{m+1}(0))$ we have  
\begin{align*}
\langle z, Rn_{m+1}(0) \rangle 
&> \frac{\|z\|^2}{2}> 0\,.
\end{align*}
Therefore $B_R(Rn_{m+1}(0))\subseteq \{z\in \R^d\colon \langle z,n_{m+1}(0)\rangle >0 \}$.

For $\zeta=y+\Phi(y)\in V\subseteq M$, we have that
\[
\langle \zeta, n_{m+1}(0)\rangle 
= 
\langle  y, L_{0,n_{m+1}} y  + r_k(y)
, n_{m+1}(0)\rangle 
=\sum_{j=1}^m \lambda_j y_j^2  + r_k(y) <0,
\]
for $\|y\|$ sufficiently small, and thus $y+\Phi(y)\notin B_R(Rn_{m+1}(0))$
for such $y$.

\medskip
\noindent{\tt Case 2:} $\lambda_1> 0$. Let $R<\frac{1}{2\lambda_1}$.
Then $\partial B_{R}(R n_{m+1}(0))$ is locally the graph of 
$n_{m+1}(0)^\perp\to \R n_{m+1}(0)$,
$y+\sum_{k=m+2}^d w_k n_k(0)\mapsto R-\sqrt{R^2 - \|y\|^2-\|w\|^2}$.
We show that the component of the manifold in direction of $n_{m+1}(0)$ 
lies below that graph, which implies that the manifold has empty intersection
with the interior of $B_{R}(R n_{m+1}(0))$ near $0$:
\begin{align*}
y+\Phi(y)
&=y+\sum_{k=m+2}^d \Phi_k(y) n_k(0)+\Phi_{m+1}(y)n_{m+1}(0)\\
&=y+  \sum_{k=m+2}^d \langle y,L_{0,n_{k}(0)} y\rangle n_k(0)+\langle y,L_{0,n_{m+1}(0)} y\rangle n_{m+1}(0)
+r(y)\,.
\end{align*}
It remains to verify that 
\begin{equation}\label{eqn:parabel}
R-\sqrt{R^2 - \|y\|^2-\sum_{k=m+2}^m\langle y,L_{0,n_{k}(0)} y\rangle^2} 
\ge \langle y,L_{0,n_{m+1}(0)}y\rangle + \tilde r(y)
\end{equation}
for $\|y\|$ sufficiently small with equality  iff $y=0$, where the remainder 
$\tilde r$ satisfies $\lim_{y\to 0}\|y\|^{-2}\tilde r(y)=0$.
It is easy to see that for $y\in \R^m$, $\|y\|$ small enough , 
\begin{align*}
R-\sqrt{R^2 - \|y\|^2-\sum_{k=m+2}^m\langle y,L_{0,n_{k}(0)} y\rangle^2} 
&\ge 
R-\sqrt{R^2 - \|y\|^2}\\
&\stackrel{(*)}{\ge} \lambda_1 \|y\|^2 \ge \sum_{k=1}^m \lambda_k y_k^2
=\langle y,L_{0,n_{m+1}(0)}y\rangle\,,
\end{align*}
with equality in $(*)$ iff $y=0$, and \eqref{eqn:parabel} follows.

Hence there exists $\varepsilon>0$ with 
$B_R(R n_{m+1})\cap M \cap B_\varepsilon(0)=\emptyset$. 

\medskip

On the other hand, if $R>\frac{1}{2\lambda_1}$, for all 
$\varepsilon>0$ 
we have $B_R(R n_{m+1})\cap M \cap B_\varepsilon(0)\ne\emptyset$. 
This is obtained by a similar calculation, where one has to note that
there exists  $a>0$ such that for all $y$ with $\|y\|$ sufficiently small 
but $\|y\|\ne 0$,
\begin{align*}
R-\sqrt{R^2 - \|y\|^2-\sum_{k=m+2}^m\langle y,L_{0,n_{k}(0)} y\rangle^2} 
&< 
R-\sqrt{R^2 - \|y\|^2}- a \|y\|^2\,.
\end{align*}
\noindent{\tt Case 3:} $\lambda_1=0$. Similar to Case 2, 
one can show that for all $R>0$ there exists $\varepsilon>0$ such that
$B_R(R n_{m+1})\cap M \cap B_\varepsilon(0)=\emptyset$. 
\end{proof}

\begin{proof}[Proof of \autoref{theorem:Dp-formula}] Fix $x\in \cE(M)$ and let $(V,n_{m+1},\ldots,n_d)$
be an orthonormal moving frame of $\nu(V)$ with $p(x)\in V$. 

\noindent  \texttt{Step 1.} 
We show that
\[
\Big(\id_{T_{p(x)}(M)}+\!\!\sum_{j=m+1}^d\!\!\big\langle x-p(x) , n_j(p(x))\big\rangle P_{T_{p(x)}(M)}(Dn_j)(p(x))\Big)Dp(x)=P_{T_{p(x)}(M)}\,.
\]
For this  consider a $C^1$ curve
$\gamma:(-\varepsilon,\varepsilon)\to\cE(M)$ with $p(\gamma(s))\in V$ for all 
$s\in(-\varepsilon,\varepsilon)$ and  $\gamma(0)=x$.  We can write
\[
\gamma=p\circ \gamma+(\gamma-p\circ \gamma)
=p\circ \gamma+\!\!\sum_{j}\!\!\big\langle \gamma-p\circ \gamma,(n_j\circ p\circ \gamma)\big\rangle (n_j\circ p\circ \gamma)\,,
\] 
where for the sake of brevity it is understood that summation ranges
from $m+1$ to $d$.
Define  curves $\beta_j:=n_j\circ p\circ \gamma$ and  scalar functions
$y_j:=\langle\gamma-p\circ \gamma, \beta_j\rangle$ for $j=m+1,\dots,d$, so that 
$\gamma=p\circ \gamma+\sum_{j}y_j \beta_j$,
and abbreviate $P=P_{T_{p(x)}(M)}$. Then
\begin{align*}
\dot \gamma&=(Dp\circ \gamma) \dot \gamma+ \sum_j(\dot y_j \beta_j+y_j\dot \beta_j)\\
P\dot \gamma&=P(Dp\circ \gamma) \dot \gamma+  \sum_j (P\dot y_j \beta_j + P y_j\dot \beta_j)
=P(Dp\circ \gamma) \dot \gamma+  \sum_j (\dot y_j P\beta_j +   y_j P\dot \beta_j)\\
&=P(Dp\circ \gamma) \dot \gamma+  \sum_j(\dot y_j P\beta_j +   y_j P(Dn_j\circ p\circ \gamma) (D p\circ \gamma) \dot \gamma)\,.
\end{align*}
Note that $Dp:T(\cE(M))\to T(M)$,
so $Dp(x)$ maps into $T_{p(x)}(M)$ and thus $PDp=Dp$.  
Since $P \beta_j(0)=P n_j\circ p\circ \gamma (0)=0$ and 
$(Dp\circ \gamma(0)) \dot
\gamma(0)\in T_{p(\gamma(0))}(M)=T_{p(x)}(M)$ and $\dot \gamma$ can be chosen
freely in $T_{p(x)}(M)$, it follows
\begin{align}
\nonumber P &=P(Dp\circ \gamma)(0) + \sum_j y_j(0) P(Dn_j\circ p\circ \gamma)(0)Dp\circ \gamma(0)\\
\nonumber &=\Big(\id_{T_{p(x)}(M)}+\sum_j\big\langle\gamma(0)-p(\gamma(0)), n_j(p(x))\big\rangle P(Dn_j\circ p\circ \gamma)(0)\Big)Dp\circ \gamma(0)\\
\label{eqn:Dp}&=\Big(\id_{T_{p(x)}(M)}+\sum_j\big\langle x-p(x), n_j(p(x))\big\rangle 
P(Dn_j) p(x)\Big)Dp(x)\,.
\end{align}
In particular, in the case where $x\in M$, we see $Dp(x)=P$.

\medskip

\noindent\texttt{Step 2.} 
We show that \[\id_{T_{p(x)}(M)}-\|x-p(x)\|L_{p(x),v} \]
is invertible for $\|x-p(x)\|\ne 0$. 

Otherwise there exists $t\in T_{p(x)}(M)$ that is mapped to 0
by $\id_{T_{p(x)}(M)}-\|x-p(x)\|L_{p(x),v}$, and therefore
\[
t= \|x-p(x)\|L_{p(x),v}t\,,
\]
so that $t$ is an eigenvector for 
$L_{p(x), v}$ corresponding to the 
positive eigenvalue $\|x-p(x)\|^{-1}$.
It follows from \autoref{theorem:curvature} that 
$\|x-p(x)\|\ge \varrho\big(p(x),\tfrac{x-p(x)}{\|x-p(x)\|}\big)$, 
which contradicts $x\in \cE(M)$ by \autoref{lemma:curvature}.

\medskip

\noindent \texttt{Step 3.} Conclusion of the proof.
For the case where $x\in M$ there is nothing left to show.

In the case where $x\notin M$ we have 
$\sum_j \langle x-p(x),n_j(p(x))\rangle n_j(p(x))=x-p(x)=\|x-p(x)\|v$ 
by the definition of $v$. Furthermore,
$v=\sum_j \langle v, n_j(p(x))\rangle n_j(p(x))$, and $(V,n)$ with 
$n:=\sum_j \langle v, n_j(p(x))\rangle n_j$
is a unit normal field
with $n(p(x))=v$.
Thus, $P(Dn)(p(x))=-L_{p(x),v}$ by 
\autoref{definition:shape}.

On the other hand, 
\begin{align*}
\|x-p(x)\|P(Dn)(p(x))&=\|x-p(x)\|P\Big(D\sum_j\langle v, n_j(p(x))\rangle n_j\Big)(p(x))\\
&=\sum_j \|x-p(x)\|\langle v, n_j(p(x))\rangle P(Dn_j)(p(x))\\
&=\sum_j \langle x-p(x), n_j(p(x))\rangle P(Dn_j)(p(x))
\end{align*}

By Steps 2 and  3, we can infer that on $T_{p(x)}(M)$
\begin{align*}
Dp(x)&=\Big(\id_{T_{p(x)}(M)}-\|x-p(x)\|L_{p(x),v}\Big)^{-1}\,,
\end{align*}
and since $Dp(x)$ vanishes on $T_{p(x)}(M)^\perp$,
\begin{align*}
Dp(x)
&=\Big(\id_{T_{p(x)}(M)}-\|x-p(x)\|L_{p(x),v}\Big)^{-1}P_{T_{p(x)}(M)}\\
&=\Big(\id_{T_{p(x)}(M)}-\|x-p(x)\|L_{p(x),v}\Big)^{-1}P_{T_{p(x)}(M)}\,.
\end{align*}
\end{proof}

\subsection*{Acknowledgments}

We are grateful to 
Joscha Prochno, 
Karin Schnass,
Peter Stadler,
Michaela Szölgyenyi, 
and Johannes Wallner for various useful comments which helped to
improve the manuscript. 


\subsection*{Author's addresses}

\noindent Gunther Leobacher, Institute of Mathematics and Scientific Computing, University of Graz. 
Heinrichstraße 36, 8010 Graz, Austria. \\{\tt gunther.leobacher@uni-graz.at}\\

\noindent Alexander Steinicke, Institute of Applied Mathematics, 
Montanuniversitaet Le\-o\-ben.
Peter-Tunner-Straße 25/I, 8700 Leoben, Austria. \\{\tt alexander.steinicke@unileoben.ac.at}
\end{document}